\documentclass[12pt]{amsart}
\usepackage{geometry}   %????????
\usepackage{enumitem} % for \setlist

\usepackage[colorlinks,citecolor = red, linkcolor=blue,hyperindex]{hyperref}
\usepackage{euscript,eufrak,verbatim, mathrsfs}
\usepackage[psamsfonts]{amssymb}
\usepackage{bbm}
\usepackage{graphicx}
\usepackage{float}
\usepackage{float, tikz}
\usepackage{orcidlink}
\usepackage{caption}
\usetikzlibrary{decorations.pathreplacing}
\usepackage{hyperref}
\usetikzlibrary{arrows.meta, positioning,fit,calc}
\usetikzlibrary{backgrounds}
\usepackage{centernot}
\usepackage{bm} %增加粗体

\usepackage{extpfeil}

\usepackage[all, cmtip]{xy}

\usepackage{upref, xcolor, dsfont}
\usepackage{amsfonts,amsmath,amstext,amsbsy, amsopn,amsthm}

\usepackage{url}

\usepackage{mathtools}
\usepackage{bookmark}

\usepackage{caption}
\usepackage{euscript}
\usepackage{times}
\usepackage{type1cm}        % activate if the above 3 fonts are
\usepackage{multicol}        % used for the two-column index
\usepackage[bottom]{footmisc}% places footnotes at page bottom
\newcommand{\normmm}[1]{{\left\vert\kern-0.25ex\left\vert\kern-0.25ex\left\vert #1 
		\right\vert\kern-0.25ex\right\vert\kern-0.25ex\right\vert}}

\usepackage{xpatch}

\newtheorem{theorem}{Theorem}[section]
\newtheorem*{theorem*}{Theorem B}
\newtheorem{lemma}[theorem]{Lemma}

\newtheorem{proposition}[theorem]{Proposition}

\newtheorem{corollary}[theorem]{Corollary}

\newtheorem*{observation*}{Observation}

\newtheorem*{assumption*}{Assumption}
\newtheorem*{question*}{Question}

\theoremstyle{definition}

\newtheorem*{definition*}{Definition}

\theoremstyle{remark}

\newtheorem*{remark*}{Remark}

\newcommand{\bt}{\boldsymbol{t}}
\newcommand{\bs}{\boldsymbol{s}}
\newcommand{\bz}{\boldsymbol{z}}
\newcommand{\bw}{\boldsymbol{w}}

\newcommand{\R}{\mathbb{R}}
\newcommand{\N}{\mathbb{N}}
\newcommand{\Z}{\mathbb{Z}}

\newcommand{\C}{\mathbb{C}}
\newcommand{\E}{\mathbb{E}}
\newcommand{\T}{\mathbb{T}}
\newcommand{\PP}{\mathbb{P}}

\newcommand{\supp}{\operatorname{supp}}

\newcommand{\GMC}{\mathrm{GMC}}

\newcommand{\an}{\text{\, and \,}}
\newcommand{\anand}{\quad\text{and}\quad}

\makeatletter
\let\oldtocsection=\tocsection
\let\oldtocsubsection=\tocsubsection
\renewcommand{\tocsection}[2]{\hspace{0em}\oldtocsection{#1}{#2}}
\renewcommand{\tocsubsection}[2]{\hspace{2em}\oldtocsubsection{#1}{#2}}
\makeatother
\makeatletter
\xpatchcmd{\@tocline}
{\hfil\hbox to\@pnumwidth{\@tocpagenum{#7}}\par}
{\ifnum#1<0\hfill\else\dotfill\fi\hbox to\@pnumwidth{\@tocpagenum{#7}}\par}
{}{}
\makeatother
\makeatletter
\renewcommand\section{\@startsection{section}{1}{\z@}%
	{-3.5ex \@plus -1ex \@minus -.2ex}%
	{2.3ex \@plus.2ex}%
	{\normalfont\large\bfseries}}%
\makeatother
\numberwithin{equation}{section}

\begin{document}
	\title[Fourier dimension of GMC on high dimensional torus]{Exact values of Fourier dimensions of Gaussian multiplicative chaos on high dimensional torus}

	\author
	{Yukun Chen}
	\address
	{Yukun Chen: School of Mathematics and Statistics, Wuhan University, Wuhan 430072, China}
	\email{yukunchen@whu.edu.cn}
	
	\author
	{Zhaofeng Lin}
	\address
	{Zhaofeng LIN: School of Fundamental Physics and Mathematical Sciences, HIAS, University of Chinese Academy of Sciences, Hangzhou 310024, China}
	\email{linzhaofeng@ucas.ac.cn}
	
	\author%[authorlabel1]
	{Yanqi Qiu}
	\address%[authorlabel1]
	{Yanqi QIU: School of Fundamental Physics and Mathematical Sciences, HIAS, University of Chinese Academy of Sciences, Hangzhou 310024, China}
	\email{yanqi.qiu@hotmail.com, yanqiqiu@ucas.ac.cn}

%\date{\today}

\begin{abstract}
We determine the exact values of the Fourier dimensions for Gaussian Multiplicative Chaos measures on the $d$-dimensional torus $\T^d$ for all integers $d \ge 1$. This resolves a problem left open in previous works \cite{LQT24,LQT25} for  high dimensions $d\ge 3$. The proof relies on a new construction of log-correlated Gaussian fields admitting specific decompositions into smooth processes with high regularity. This construction enables a multi-resolution analysis to obtain sharp local estimates on the measure's Fourier decay. These local estimates are then integrated into a global bound using Pisier's martingale type inequality for vector-valued martingales.
\end{abstract}

	\subjclass[2020]{Primary 60G57, 42A61, 46B09; Secondary 60G46}
	%46B09 Probabilistic methods in Banach space theory
	%60G57 random measures
	%60J80 Branching processes
	%60G46 martingales and classical analysis
	%42A61 Probability methods for one variable harmnoic analysis
	\keywords{Gaussian multiplicative chaos;  Polynomial Fourier decay; Fourier dimension; Vector-valued martingale method}

	\maketitle
	\tableofcontents

	\setcounter{tocdepth}{2}
	%\tableofcontents

	\setcounter{tocdepth}{0}
	\setcounter{equation}{0}

	%\setcounter{secnumdepth}{2}
	%\setcounter{tocdepth}{2}

	%\vspace{0.1in}
	
\section{Introduction}\label{sec:intro}

Gaussian Multiplicative Chaos (GMC) is a central object in modern probability theory and mathematical physics. Originating from Kolmogorov--Obukhov model of turbulence, it has since found profound connections to various fields, including the theory of random surfaces in Liouville quantum gravity, statistical physics, and financial mathematics~\cite{RV14}. A GMC measure is a random measure constructed by exponentiating a log-correlated Gaussian field, and understanding its fine geometric and analytic properties is a subject of intense research.

A key question in this area is to characterize the fractal nature of GMC measures. The Fourier dimension, which is determined by the polynomial decay rate of the measure's Fourier transform, provides a crucial insight into its regularity. We confirm that the general phenomenon
\begin{center}
\textit{The Fourier dimension of the random measure coincides with its correlation dimension.}
\end{center}
holds for GMC on high dimensional torus. For GMC on one dimensional torus, this phenomenon was conjectured by Garban--Vargas in~\cite{GV23}.
	
Recent works by Lin, Qiu and Tan~\cite{LQT24,LQT25} made significant progress by verifying this phenomenon for the classical GMC measure on $[0,1)^d$ in low dimensions $d=1, 2$ and limited parameter regime for $d\ge 3$.
However, the methods therein encountered technical obstacles that prevented a full generalization to higher dimensional torus. The dimensional barrier stemmed from the decomposition of the log-correlated Gaussian field and the dyadic approximation method used there; specifically, the resulting random processes from that decomposition lacked sufficient regularity for the required analysis to hold in a high-dimensional setting. 

In this paper, by a well designed smooth decomposition of log-correlated Gaussian field and smooth partitions of unity on $\T^d$ adapted to dyadic structure, we overcome this dimensional barrier and provide the exact value of the Fourier dimension for a GMC measure on the $d$-dimensional torus for all $d \ge 1$. Our main result is  the following Theorem~\ref{thm:Fourier-decay-GMC}. 

The novel method used in this paper is significant not only in that it unifies the treatments of Fourier decay of GMC on torus of all dimensions ($d\ge 1$) but also in that it will enable us to deal with some ``abstract-Fourier-type decay'' of GMC measures on general compact Riemannian manifolds, where the Fourier-coefficients will be replaced by the avarage of  the $L^2$-normalized eigenfunctions of the Laplace-Beltrami operator, see \S \ref{sec-loc-es} for a brief explanation. This is the subject of our forthcoming work.

\subsection{Main results}

Let $\mu$ be a Borel measure on $\T^d$. The Fourier dimension of $\mu$ is defined by
\begin{equation}\label{eq:def-Fourierdim}
\dim_F(\mu)\coloneq \sup\big\{s\in [0,d):|\widehat{\mu}(\boldsymbol{n})|^2 =  O(|\boldsymbol{n}|^{-s}) ~\text{as}~|\boldsymbol{n}|\to\infty\big\}.
\end{equation}

We denote by $d_{\T^d}(\bz,\bw)$ the standard metric on $\T^d$, whose precise definition is in \eqref{eq:def-distance-Td} below. As usual, for all $x\in (0, \infty)$, we write $\log_+(x)=\max\{0,\log x\}$.
\begin{theorem}\label{thm:Fourier-decay-GMC}
	For any integer $d\ge 1$, there exists a centered log-correlated Gaussian field on $\T^d$ with covariance kernel
	\[
	K(\bz,\bw)=\log_+\frac{1}{d_{\T^d}(\bz,\bw)}+g(\bz\bar{\bw}),\quad \forall \bz\in \T^d\setminus\{\boldsymbol{1}\},
	\] 
	where $g$ is a bounded continuous function on $\T^d$, such that for any $\gamma\in (0,\sqrt{2d})$, almost surely, the sub-critical GMC measure $\GMC_K^\gamma$ has Fourier dimension
	\[
	\dim_F(\GMC_K^\gamma)= D_{\gamma,d}\coloneq \left\{
	\begin{array}{cl}
		d-\gamma^2, & \text{if  $0<\gamma<\sqrt{2d}/2$},
		\vspace{2mm}
		\\
		(\sqrt{2d}-\gamma)^2, & \text{if $\sqrt{2d}/2\leq\gamma<\sqrt{2d}$}.
	\end{array}\right.
	\]
\end{theorem}

\begin{remark*}
We choose the torus $\mathbb{T}^d$ over the unit cube $[0,1)^d$ in \cite{LQT24,LQT25} to avoid boundary effects. A sharp truncation, for instance restricting a measure to the cube, changes its Fourier decay. For example, the Fourier coefficients of the Lebesgue measure on the cube decays as $|\boldsymbol{n}|^{-1}$ as $|\boldsymbol{n}|\to\infty$. This boundary effect may explain the lower bound restriction $\dim_F(\mu_\infty) \ge \min\{2, D_{\gamma,d}\}$ for the analysis on the cube in \cite[Theorem~1.2]{LQT25}.

Indeed, if we replace the sharp cutoff with a smooth truncation function that is compactly supported in the interior of $(0,1)^d$, then the strategy in this paper could also be adapted to the unit cube. 

The torus, being free of boundaries, provides a more direct and natural setting for our analysis.
\end{remark*}

\begin{remark*}
Recall that the correlation dimension of a measure $\mu$ is defined by (see, e.g.,~\cite[Lemma~2.6.6 and Definition~2.6.7]{BSS23})
\[
\dim_2(\mu)\coloneq \liminf_{\delta\to 0^+}\frac{\log \big(\sup\sum_i \mu(B(\boldsymbol{x}_i,\delta))^2\big)}{\log(1/\delta)},
\]
where the supremum is taken over all family of disjoint balls. It is a known result that the correlation dimension is given by $D_{\gamma,d}$ (see, e.g.,~\cite[Theorem~3.1 and Formula(3.2)]{Ber23}, \cite[Section~4.2]{RV14} and \cite[Remark~2]{GV23}; see also \cite[Lemma~3.6]{LQT25} for a precise formulation):  if the kernel
\[
K(\bz,\bw)=\log_+ \frac{1}{d_{\T^d}(\bz,\bw)}+g(\bz,\bw)
\]
has bounded continuous remainder $g$, then for any $\gamma\in (0, \sqrt{2d})$, almost surely, 
\[
\dim_2(\GMC_K^\gamma)=D_{\gamma,d}.
\]
Therefore, we can reformulate Theorem~\ref{thm:Fourier-decay-GMC} by stating that there exists a log-correlated Gaussian field on $\T^d$, such that the corresponding sub-critical GMC measure has the same Fourier dimension as its correlation dimension.
\end{remark*}

\subsection{Outline of the proof}
In this section, we give a broad overview of our strategy to prove Theorem~\ref{thm:Fourier-decay-GMC} by Pisier's martingale type-$p$ inequality and a multi-resolution framework. We illustrate the logical flow of the proof of Theorem~\ref{thm:Fourier-decay-GMC} in Figure~\ref{fig:proof_flow}.
\begin{figure}[!htbp]
	\centering
	\begin{tikzpicture}[
		% 定义垂直和水平的节点间距
		node distance=1.5cm and 2cm,
		% 沿用 Figure 1 的样式定义
		every node/.style={font=\small, align=center},
		box/.style={draw, rounded corners, minimum width=3.2cm, minimum height=1.3cm, fill=gray!10},
		arrowgreen/.style={-{Latex[length=3mm]}, thick, shorten >=4pt, shorten <=4pt, color=green},
		arrowred/.style={-{Latex[length=3mm]}, thick, shorten >=4pt, shorten <=4pt, color=red},
		arrowblue/.style={-{Latex[length=3mm]}, thick, shorten >=4pt, shorten <=4pt, color=blue},
		arrowgray/.style={-{Latex[length=3mm]}, thick, shorten >=4pt, shorten <=4pt, color=gray},
		desc/.style={midway, fill=white, inner sep=2pt}  
		]
		
		% --- 节点定义 (根据您的草图布局) ---
		% 第一行节点
		\node[box] (prop31) {Proposition \ref{prop:smooth-decomposition-log-correlated-field}\\(Smooth Decomposition)};
		\node[box, right= of prop31] (lemma46) {Lemma~\ref{lemma:p-moment-higher-derivatives-Xj}\\(Moment Estimates \\of Derivatives)};
		\node[box, below=0.7cm of lemma46] (lemma45) {Lemma \ref{lemma:p-moment-Xj}\\(Moment Estimates)};
		\node[box, right=of lemma46] (prop44) {Proposition \ref{prop:localization}\\(Local Estimate)};
		
		% 定理节点，位于中间下方
		\node[box, below=3cm of prop31] (thm11) {Theorem \ref{thm:Fourier-decay-GMC}\\(Fourier Dimension of GMC)};
		% Prop 1.3 作为 Prop 3.2 和 Thm 1.1 之间的桥梁
		\node[box, below=3 cm of prop44] (prop12) {Proposition \ref{prop:uniform-boundedness-Fourier-Lebesgue-norm}\\($p$-moment of $\mathcal{F}L^{\tau/2,q}$-norm)};

		% --- 箭头连接 (根据我们之前确定的逻辑) ---
		% 线性逻辑链
		\draw[arrowred] (prop31) -- (lemma46);
		\draw[arrowgray] (prop31) -- (lemma45);
		\draw[arrowgray] (lemma45) -- (prop12);
		\draw[arrowred] (lemma46) -- (prop44);
		\draw[arrowred] (prop44) -- (prop12);
		\path (prop44) -- (prop12) node[desc, pos=0.5, right=-2.5cm] {integrate by Pisier's martingale\\ type-$p$ inequality};
		
		\draw[arrowblue] (prop12) -- (thm11);
		\path (prop12) -- (thm11) node[desc, pos=0.5, below=0.3cm] {implies lower bound of $\dim_F(\mu_\infty)$};
		\draw[arrowblue] (prop31) -- (thm11);
		\path (prop31) -- (thm11) node[desc, pos=0.5, right=-1.5cm] {implies upper bound \\of $\dim_F(\mu_\infty)$};
	\end{tikzpicture}
	\caption{Logical flow of the proof of Theorem~\ref{thm:Fourier-decay-GMC}. The color scheme distinguishes the roles of different components in the argument: red highlights the core technical path for establishing the lower bound; blue shows how to obtain the upper and lower bounds for the final theorem; and gray represents supporting derivations.}
	\label{fig:proof_flow}
\end{figure}
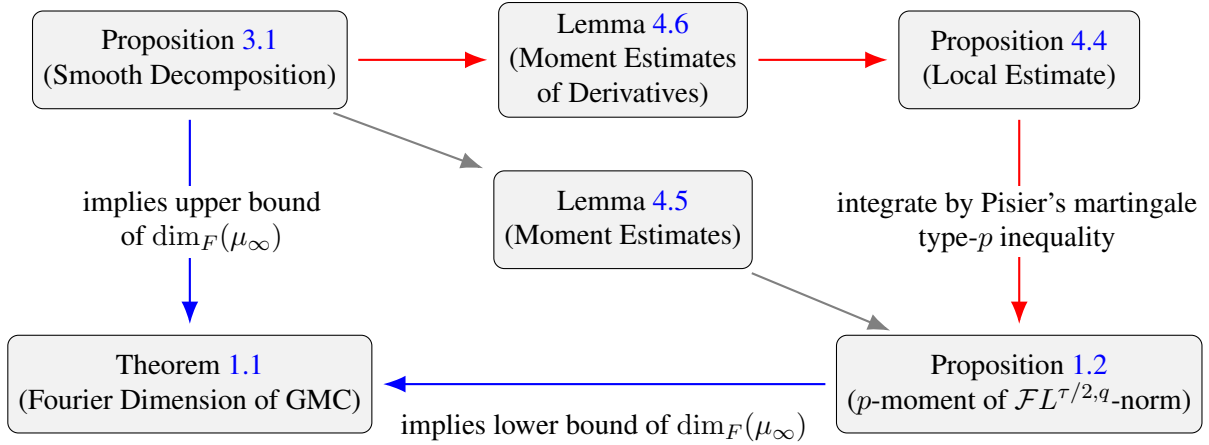

We first give a brief introduction to the definition of GMC measure. Let $\mathcal{K}(\bz,\bw)$ be a positive-definite kernel on $\T^d$ with
\[
K(\bz,\bw)=\log_+ \frac{1}{d_{\T^d}(\bz,\bw)}+g(\bz,\bw),\quad \forall \bz,\bw\in \T^d\setminus \{\bz=\bw\},
\]
where $g$ is a bounded function on $\T^d\times \T^d$. Following Kahane~\cite{Kah85}, if $K$ can be decomposed as
\begin{equation}\label{eq:sigma-positivity}
K(\bz,\bw)=\sum_{j=1}^\infty K_{j}(\bz,\bw),
\end{equation}
where $\mathcal{K}_j$ is continuous and non-negative definite kernel, then we say that $K$ is \textit{$\sigma$-positive}. Let $\psi_j$ be a centered Gaussian process generated by $K_j$. Define
\begin{equation}\label{eq:def-Xj}
\Big\{X_j(\bz)=\exp\big(\gamma\psi_j(\bz)-\frac{\gamma^2}{2}\E[\psi_j(\bz)^2]\big):\bz\in \T^d\Big\}.
\end{equation}
Let
\[
\mu_k(\mathrm{d}\bz)\coloneq \prod_{j=1}^k X_j(\bz)\,\boldsymbol{m}(\mathrm{d}\bz).
\]
According to Kahane's $T$-martingale theory~\cite{Kah87}, the sequence of random measures $\{\mu_k\}$ weakly converges to a limit random measure $\mu_{\infty}$ almost surely. If $\gamma\in (0,\sqrt{2d})$, then the limit measure $\mu_\infty$ is non-degenerate, that is, $\mu_\infty(\T^d)\ne 0$, a.s. A key result from Kahane is that this limit measure is also independent of the decomposition in \eqref{eq:sigma-positivity}. The GMC measure $\GMC_K^\gamma$ in Theorem~\ref{thm:Fourier-decay-GMC} is then defined by
\[
\GMC_K^\gamma\coloneq \mu_\infty.
\]
Henceforth, we shall simply use the notation $\mu_\infty$ rather than $\GMC_K^\gamma$.

The upper bound of $\dim_F(\mu_\infty)$ is given
\[
\dim_F(\mu_\infty)\le \dim_2(\mu_\infty),
\]
while the exact value of correlation dimension $\dim_2(\mu_\infty)$ was known to be $D_{\gamma,d}$ (see Bertacco \cite[Theorem~3.1 and Formula~(3.2)]{Ber23}, Rhodes--Vargas \cite[Section~4.2]{RV14}, Garban--Vargas \cite[Remark~2]{GV23} and Lin--Qiu--Tan \cite[Lemma~3.6]{LQT25}). 

The main difficulty lies in proving the matching lower bound, $\dim_F(\mu_\infty) \ge D_{\gamma,d}$. Our strategy is to prove that the GMC measure almost surely lies in a certain function space called the (weighted) \textit{Fourier--Lebesgue space}.  Denote the space of Schwartz distributions on $\T^d$ by $\mathcal{D}'(\T^d)$. Let $s\in \R$ and $1\le q<\infty$. The Fourier--Lebesgue space is defined as
\[
\mathcal{F}L^{s,q}(\T^d)\coloneq \Big\{f\in \mathcal{D}'(\T^d):\|f\|_{\mathcal{F}L^{s,q}(\T^d)}<\infty\Big\}, 
\]
with the norm $\|\cdot\|_{\mathcal{F}L^{s,q}(\T^d)}$ defined by
\begin{equation}\label{eq:def-FLnorm}
\|f\|_{\mathcal{F}L^{s,q}(\T^d)}\coloneq \Big\{\sum_{\boldsymbol{n}\in \Z^d} \langle\boldsymbol{n}\rangle^{s q}|\widehat{f}(\boldsymbol{n})|^q\Big\}^{1/q},\quad 1\leq q<\infty,
\end{equation}
where \(\langle\boldsymbol{n}\rangle\coloneq (1+|\boldsymbol{n}|^2)^{1/2}\) and
\[
|\boldsymbol{n}|\coloneq \Big(\sum_{i=1}^d |n_i|^2\Big)^{1/2},\quad \forall \boldsymbol{n}=(n_1,n_2,\ldots,n_d)\in \Z^d.
\]
 Note that once we can prove $f\in \mathcal{F}L^{s,q}$ for some $q\ge 1$, then we immediately have $|\widehat{f}(\boldsymbol{n})|=o(|\boldsymbol{n}|^{-s})$. This is why we can use this space to detect the Fourier decay of GMC measure. 

The Fourier--Lebesgue space arose in the context of time-frequency analysis. It was shown in~\cite[Proposition~2.1 and Remark 4.2]{RSTT11} that the Fourier--Lebesgue space coincides with classical space like modulation space and Wiener amalgam space. For those who are interested in the connection of probability and these spaces, we refer to~\cite{BO11}, where they study the regularity of Brownian motion in modulation spaces.

The lower bound of Fourier dimension of the GMC measure \(\mu_\infty\) is given by the following proposition.
\begin{proposition}\label{prop:uniform-boundedness-Fourier-Lebesgue-norm}
For any $\tau\in (0,D_{\gamma,d})$, there exist $p$ and $q$ satisfying $1<p\le 2\le q<\infty$, such that
\[
\mathbb{E}\big[\|\mu_\infty\|_{\mathcal{F}L^{\tau/2,q}}^p\big]=\sup_{m\ge 1}\E\big[\|\mu_m\|_{\mathcal{F}L^{\tau/2,q}}^p\big]=\sup_{m\ge 1}\E\Big[\Big\{\sum_{\boldsymbol{n}\in \Z^d} \langle\boldsymbol{n}\rangle^{\tau q/2}|\widehat{\mu}_m(\boldsymbol{n})|^q\Big\}^{p/q}\Big]<\infty.
\]
\end{proposition}

To prove Proposition~\ref{prop:uniform-boundedness-Fourier-Lebesgue-norm}, the core idea is to first decouple the problem across multiple scales and locations, reducing it to a local estimate, and then to establish this local estimate using new analytic techniques.

\subsubsection{Random decoupling scheme: vector-valued martingale method}

Our proof unfolds within a \textit{multi-resolution framework}, which naturally mirrors the inherent scale-by-scale, multiplicative construction of the GMC measure. This framework allows us to exploit the inherent martingale structure of GMC, which is a cornerstone of further application of \textit{Pisier's martingale type inequality}. By applying Pisier's martingale type-$p$ inequality twice, we can decouple the global estimate of $\E[\|\mu_m\|_{\mathcal{F}L^{\tau/2,q}}^p]$ into localized pieces with various scales and locations. See \S\ref{sssec:Pisier-mar-p-ineq} for a short introduction to this inequality and Step 1 and Step 2 in \S\ref{ssec:proof-uniform-bdn} for its application in the proof of Proposition~\ref{prop:uniform-boundedness-Fourier-Lebesgue-norm}.

This method is called \textit{vector-valued martingale method}, which is initially introduced in~\cite{CHQW24} in the setting of Mandelbrot cascade and later applied on one-dimensional GMC measure in~\cite{LQT24} and several models including high dimensional GMC on $[0,1)^d$ in sequential work~\cite{LQT25}. 

Let us make the argument slightly precise. Denote the $k$-generation dyadic cubes in $[-1/2,1/2)^d$ by $\mathscr{D}_k$ (see \eqref{eq:def-D_j} for its precise definition). Exploiting the inherent martingale structure of $\mu_m$ and applying Pisier's martingale type-$p$ inequality twice, we can decouple the $p$-moment of $\mathcal{F}L^{\tau/2,q}$ into sum of ``localized Fourier--Lebesgue norms'' of $\mu_m$, that is,
\[
\mathbb{E}\big[\|\mu_m\|_{\mathcal{F}L^{\tau/2,q}}^p\big]\lesssim \sum_{k=1}^m \sum_{\boldsymbol{I}\in \mathscr{D}_k}\mathbb{E}[\|\mathfrak{D}_{\boldsymbol{I}}^\varphi\|_{\ell^q}^p],
\]
where
\begin{equation}\label{eq:intro-def-DIphi}
	\mathfrak{D}_{\boldsymbol{I}}^\varphi(\boldsymbol{n})\coloneq \langle \boldsymbol{n}\rangle^{\tau/2}\int_{\T^d}\varphi_{\boldsymbol{I}}(\bz)\bigg[\prod_{j=1}^{k-1}X_j(\bz)\bigg]\cdot(X_k(\bz)-1)\cdot \bz^{\boldsymbol{n}}\,\boldsymbol{m}(\mathrm{d}\bz)
\end{equation}
and the support of $\varphi_{\boldsymbol{I}}$ is concentrated in a small neighbourhood of $\boldsymbol{I}$ (see \S\ref{sssec:pou} for its precise statement). Therefore, we may intuitively regard the term $\|\mathfrak{D}_{\boldsymbol{I}}^\varphi\|_{\ell^q}$ as a local analogue of the Fourier-Lebesgue norm of $\mu_m$.

\subsubsection{Localization estimate}\label{sec-loc-es}

While the vector-valued martingale method provides the global scheme for the proof, its success relies on a sharp enough estimate for $\mathbb{E}[\|\mathfrak{D}_{\boldsymbol{I}}^\varphi\|_{\ell^q}^p]$.  This is the content of Proposition~\ref{prop:localization} and represents the main technical innovation of this paper, allowing us to overcome the dimensional barrier encountered in previous works.

The core of this analysis lies in estimating the Fourier decay of the localized measure. A standard technique for estimating such oscillatory integrals is repeated integration by parts in \eqref{eq:intro-def-DIphi}, which we employ here via Green's identity in  \S\ref{sssec:Bound-Ihigh}. Our choice to center this technique on the Laplacian is deliberate and twofold. 
\begin{enumerate}
	\item First, the Laplacian is the natural operator for this task, as each application generates a decay factor $|\boldsymbol{n}|^{-2}$ in the frequency domain, directly corresponding to the structure of the Fourier-Lebesgue norm \eqref{eq:def-FLnorm} and Fourier dimension \eqref{eq:def-Fourierdim}.
	\item Second, our analysis is fundamentally based on using the eigenfunctions of the relevant Laplacian on $\T^d$, this is the Fourier basis $\{\bz^{\boldsymbol{n}}\}_{\boldsymbol{n}\in \Z^d}$. This principle extends directly to a general compact Riemannian manifold, where the analysis would rely on the eigenfunctions of the Laplace-Beltrami operator. Indeed, extending this analysis to general compact Riemannian manifolds without boundary is the subject of our forthcoming work.
\end{enumerate}

However, for this classical analytical tool to be effective in our probabilistic setting, the random function being analyzed (that is, the product of our processes $\{X_j\}_{j\in \mathbb{N}\setminus \{0\}}$) and the localizing window functions must have sufficient regularity ($C^\infty$ for example) to withstand repeated differentiation in iterated integration by parts.

Direct application of previous methods in \cite{LQT25} fails because they could not meet both regularity conditions. The following two central technical advancements are designed in this work.
\begin{enumerate}
	\item First, we introduce a log-correlated Gaussian field which admits a decomposition into smooth processes of high regularity by a designed convolution argument. This is the cornerstone of our analysis. The regularity of paths $\{X_j\}_{j\in \mathbb{N}\setminus \{0\}}$ admit the application of integration by parts to gain faster decay. This is the theme of \S\ref{sec:smooth-decomposition}.
	\item Second, our multi-resolution analysis relies on a careful localization technique using a smooth partition of unity $\{\varphi_{\boldsymbol{I}}\}_{\boldsymbol{I}\in \mathscr{D}_k}$. This is a deliberate choice over a simpler truncation with characteristic functions. In~\cite{LQT25}, the sharp spatial cutoff introduces significant spectral leakage, which would distort the frequency content of the localized measure and prevent a sharp estimate. The partition of unity is introduced in Lemma~\ref{lemma:pou-scaling-Td} and it is the building block of the localization estimate Proposition~\ref{prop:localization}. 
\end{enumerate}

We borrow the term ``spectral leakage'' from signal processing (see, e.g., \cite[Example~10.3]{OS10} and \cite{Har78}) to emphasize the intuition that direct truncation considerably disturbs the Fourier spectrum. This can be seen even in the simplest deterministic case. Consider the constant function $f(z) = 1$ for all $z \in \mathbb{T}\coloneq\{z\in \C:|z|=1\}$. Its Fourier series is trivial, with coefficients $\widehat{f}(0) = 1$ and $\widehat{f}(n) = 0$ for all $n \neq 0$. This represents a perfect spectrum with no leakage.

Now, let's observe the effect of applying a window function to $f(z)$. We compare two cases:
\begin{enumerate}
	\item \textbf{A sharp window:} Let $g_1(z)$ be the characteristic function of the arc $\{e^{2\pi i t}:t\in [-1/4, 1/4]\}$. To compute its $n$-th Fourier coefficient, we use the parameterization $z=e^{2\pi it}$:
	\[
	\widehat{f\cdot g_1}(n) = \int_{\mathbb{T}} g_1(z) z^n\,m(\mathrm{d}z) = \int_{-1/4}^{1/4} e^{2\pi i nt}\,\mathrm{d}t = \frac{\sin(\pi n/2)}{\pi n}.
	\]
	This exhibits a slow algebraic decay of order $|n|^{-1}$ as $|n|\to \infty$, which is a poor approximation of the true spectrum (which is zero).
	
	\item \textbf{A smooth window:} Let $g_2(z)$ be any $C^\infty$ function on $\mathbb{T}$ supported on the same arc. Its $n$-th Fourier coefficient is
	\[
	\widehat{f\cdot g_2}(n) = \int_{\mathbb{T}} g_2(z) z^n\,m(\mathrm{d}z)=\int_{-1/2}^{1/2}g(e^{it})e^{2\pi i nt}\,\mathrm{d}t.
	\]
	Through repeated integration by parts, we can see that its coefficients decay faster than any polynomial, that is, $|\widehat{g_2}(n)| = O(|n|^{-k})$ as $|n|\to \infty$ for any $k \in \mathbb{N}$.
\end{enumerate}
The rapid decay of $\widehat{g_2}(n)$ is a far better approximation of the zero-valued spectrum of the original signal $f(z)=1$. This illustrates how a smooth window better preserves the underlying spectral properties.

It is the combination of this localization technique that effectively suppresses spectral leakage and the high regularity of our underlying processes that ultimately allows the global martingale argument to succeed.

\subsection{Organization of the paper}
In the Introduction \S\ref{sec:intro}, we present the background on Gaussian Multiplicative Chaos, state our main result on the exact Fourier dimension of GMC measures on the torus (Theorem~\ref{thm:Fourier-decay-GMC}), and outline the proof strategy.

In \S\ref{sec:notations-preliminaries}, we introduce the necessary notations and preliminary tools for our analysis, including Pisier's martingale type inequality and a smooth partition of unity on the torus.

In \S\ref{sec:smooth-decomposition}, we construct a log-correlated Gaussian field that admits a decomposition into highly regular, smooth processes and establish its crucial properties in Proposition~\ref{prop:smooth-decomposition-log-correlated-field}.

In \S\ref{sec:proof-main-thm}, we prove the main theorem by establishing a key moment bound for the Fourier-Lebesgue norm of the GMC measure (Proposition~\ref{prop:uniform-boundedness-Fourier-Lebesgue-norm}). 

In Appendix~\ref{appendix:pou}, we provide the detailed construction of the smooth partition of unity (Lemma~\ref{lemma:pou-scaling-Td}) used in the multi-resolution analysis in \S\ref{sec:proof-main-thm}.

\subsection*{Acknowledgements}  YQ is supported by National Natural Science Foundation of China (NSFC No. 12471145).

\section{Notations and preliminaries}\label{sec:notations-preliminaries}

\subsection{Notations}

Let $d\ge 1$ be an integer. We identify $d$-dimensional torus $\T^d$ with $(S^1)^d\subset \mathbb{C}^d$ by
\[
\T^d\coloneq\{(z_1,z_2,\ldots,z_d)\in \C^d:|z_i|=1,i=1,2,\ldots,d\}.
\]
The group operation on $\T^d$ is coordinate-wise multiplication, that is, for $\bz=(z_1,z_2,\ldots,z_d)$ and $\bw=(w_1,w_2,\ldots,w_d)$ in $\T^d$, 
\[
\bz\cdot \bw\coloneq (z_1w_1,z_2w_2,\ldots,z_d w_d).
\]
We also write $\bz\bw$ instead of $\bz\cdot \bw$ for simplicity. The unit of multiplication on $\T^d$ is denoted by $\boldsymbol{1}$.

As usual, we denote
\[
e(t)=e^{2\pi i t},\quad \forall t\in \R.
\]
For each $z_i\in \T$, there exists a unique $t_i\in [-1/2,1/2)$, such that $z_i=e(t_i)$. For convenience, we define
\begin{align*}
\begin{array}{cccc}
\mathcal{E}\colon & \R^d& \xrightarrow{\quad\quad} & \T^d
\vspace{2mm}
\\
& (t_1,t_2,\ldots,t_d) & \mapsto& (e(t_1),e(t_2),\ldots,e(t_d)).
\end{array}
\end{align*}
Therefore, we can parameterize $\T^d$ as
\begin{align}\label{Td-para}
\T^d=\big\{\mathcal{E}(\bt)\in \C^d:\bt\in [-1/2,1/2)^d\big\}.
\end{align}

For each point $\bz\in \T^d$, define
\[
\|\bz\|_{\T^d}\coloneq \inf_{\mathcal{E}(\bt)=\bz}\|\bt\|,
\]
where $\|\bt\|$ is the Euclidean norm on $\R^d$, that is,
\[
\|\bt\|=\bigg(\sum_{j=1}^d |t_j|^2\bigg)^{1/2},\quad \forall \bt=(t_1,t_2,\ldots,t_d)\in \R^d.
\]
If $\bz=\mathcal{E}(\bt)$ for some $\bt\in [-1/2,1/2)^d$, then we have the explicit expression
\[
\|\bz\|_{\T^d}=\bigg(\sum_{j=1}^d |t_j|^2\bigg)^{1/2}.
\]
This induces the metric $d_{\T^d}$ on $\T^d$ by 
\begin{equation}\label{eq:def-distance-Td}
d_{\T^d}(\bz,\bw)\coloneq\|\bz\bar{\bw}\|_{\T^d},\quad \forall \bz,\bw\in \T^d,
\end{equation}
where $\bar{\bw}=(\bar{w}_1,\bar{w}_2,\ldots,\bar{w}_d)$ is the complex conjugate of $\bw$. For $r>1$ and $\bz\in \T^d$, we denote
\[
\overline{B}_{\T^d}(\bz,r)\coloneq \{\bw\in \T^d:d_{\T^d}(\bz,\bw)\le r\}.
\]
While we reserve the symbol $\overline{B}(\bt,r)$ for closed ball in $\R^d$, that is,
\[
\overline{B}(\bt,r)\coloneq\{\bs\in \R^d:\|\bt-\bs\|\le r\}.
\]

Denote the family of $j$-th generation of dyadic cubes in $\R^d$ by
\[
\widetilde{\mathscr{D}}_{j}^d=\widetilde{\mathscr{D}}_{j}\coloneq \bigg\{[0,2^{-j})^d+2^{-j}\boldsymbol{h}: \boldsymbol{h}\in \Z^d\bigg\}
\]
and the family of $j$-th generation of dyadic cubes in $[-1/2,1/2)^d$, that is, the Euclidean coordinate of $\T^d$, by
\begin{equation}\label{eq:def-D_j}
\mathscr{D}_{j}^d=\mathscr{D}_{j}\coloneq \bigg\{[0,2^{-j})^d+2^{-j}\boldsymbol{h}: \boldsymbol{h}\in \big\{-2^j,-2^j+1,\ldots,2^{j}-1\big\}^d\bigg\}.
\end{equation}
For any cube $\boldsymbol{I} \subset \R^d$ and scalar $\lambda > 0$, 
we denote by $\lambda \boldsymbol{I}$ the cube concentric with $\boldsymbol{I}$ 
and scaled by a factor of $\lambda$. 

Let $\boldsymbol{\alpha}= (\alpha_1,\alpha_2,\ldots,\alpha_d)\in \N^d$ be a multi-index. The length $|\boldsymbol{\alpha}|$ of $\boldsymbol{\alpha}$ is defined by 
\[
|\boldsymbol{\alpha}|\coloneq \alpha_1+\alpha_2+\cdots+\alpha_d.
\]

For any sufficiently smooth function $f\colon \T^d\to \mathbb{C}$  written as 
\[
(e(t_1),e(t_2),  \ldots,  e(t_d))\mapsto f(e(t_1),e(t_2),\ldots,e(t_d))
\]
and  any multi-index $\boldsymbol{\alpha}=(\alpha_1,\alpha_2,\ldots,\alpha_d)$, define
\[
(D^{\boldsymbol{\alpha}} f)(e(t_1),e(t_2),\ldots,e(t_d))\coloneq (D_1^{\alpha_1}D_2^{\alpha_2}\cdots D_d^{\alpha_d}f)(e(t_1),e(t_2),\ldots,e(t_d)),
\]
where $D_i$ denotes the first order partial derivative of $f$ with respect to $t_i$, that is,
\[
(D_i f)(e(t_1),e(t_2),\ldots,e(t_d))\coloneq\lim_{h\to 0}\frac{f(e(t_1),\ldots,e(t_i+h),\ldots,e(t_d))-f(e(t_1),\ldots,e(t_i),\ldots,e(t_d))}{h}.
\]
On the other hand, if  $g \colon U\to \R,\bt\mapsto g (\bt)$ is a sufficient smooth function on a domain  $U\subset \R^d$, to distinguish from the partial derivatives on $\T^d$, we denote the partial derivative of $g$ with respect to the $i$-th coordinate by $\partial_i g$ and 
\[
(\partial^{\boldsymbol{\alpha}} g)(\bt)\coloneq (\partial_1^{\alpha_1}\partial_2^{\alpha_2}\cdots \partial_d^{\alpha_d}g)(\bt).
\] 

We also denote by  $\boldsymbol{m}$ the normalized Haar measure on $\T^d$. To be specific, for any $f\in L^1(\T^d)$,
\begin{align*}
\int_{\T^d} f(\bz)\,\boldsymbol{m}(\mathrm{d}\bz)&=\int_{[-1/2,1/2)^d}f(e(t_1),e(t_2),\ldots,e(t_d))\,\mathrm{d}t_1\mathrm{d}t_2\cdots\mathrm{d}t_d 
\\
& =\int_{[-1/2,1/2)^d}f(\mathcal{E}(\bt))\,\mathrm{d}\bt,
\end{align*}
where $\mathrm{d}\bt\coloneq \mathrm{d}t_1\mathrm{d}t_2\cdots\mathrm{d}t_d$ is the Lebesgue measure on $[-1/2,1/2)^d$.

For any integer $k \ge 1$ and any $C^k(\R^d)$ function $f$, its $C^k$-norm is defined by
\[
\|f\|_{C^k(\R^d)}\coloneq \sum_{|\boldsymbol{\alpha}|\le k}\|\partial^{\boldsymbol{\alpha}}f\|_{L^\infty(\R^d)}.
\]

We also use the standard notation 
\[
e_{\boldsymbol{n}}(\bt)\coloneq e\bigg(\sum_{\ell=1}^d n_\ell t_\ell\bigg),\quad \forall \bt=(t_1,t_2,\ldots,t_d)\in \R^d \an \boldsymbol{n}\in \Z^d.
\]
The Fourier transform of a function $f\in L^1(\T^d)$ is defined by\footnote{We remind the reader that our definition of the Fourier transform differs slightly from the conventional one, in which $\widehat{f}(\boldsymbol{n})=\int_{\T^d} f(\bz)\bar{\bz}^{\boldsymbol{n}}\,\boldsymbol{m}(\mathrm{d}\bz)$. We choose this different convention omitting the complex conjugate for simplicity of notation.}
\[
\widehat{f}(\boldsymbol{n})\coloneq \int_{\T^d} f(\bz)\bz^{\boldsymbol{n}}\,\boldsymbol{m}(\mathrm{d}\bz)=\int_{[-1/2,1/2)^d}f(\mathcal{E}(\bt))e_{\boldsymbol{n}}(\bt)\,\mathrm{d}\bt,\quad \forall \boldsymbol{n}\in \Z^d.
\]

\subsection{Key technical tools}

\subsubsection{Mollification of random processes}

Our construction of the log-correlated Gaussian field with smooth decomposition relies on a mollification process by convolution. This is a standard method in deterministic case.
\begin{lemma}\label{lemma:convolution-deterministic}
	Let $f\in L^1(\T^d)$ and $P\in C^k(\T^d)$ for some $k\in \N\cup\{\infty\}$. Then the convolution of $f$ and $P$
	\[
	(P*f)(\bz)=\int_{\T^d} P(\bz\bar{\bw})f(\bw)\,\boldsymbol{m}(\mathrm{d}\bw),\quad \forall \bz \in \T^d
	\]
	is a $C^k$-function and for any multi-index $\boldsymbol{\alpha}$ with $|\boldsymbol{\alpha}|\le k$, we have
	\begin{equation}\label{eq:diff-convolution}
	D^{\boldsymbol{\alpha}}(P*f)(\bz)=((D^{\boldsymbol{\alpha}} P)*f)(\bz),\quad \forall \bz\in \T^d.
	\end{equation}
\end{lemma}
\begin{proof}
	The proof of $\R^d$ version can be found in~\cite[Proposition~4.20]{Bre11}. The same proof applies verbatim to $\T^d$, since the argument relies only on local properties.
\end{proof}

Now, we can extend this result to Gaussian processes with continuous paths.
\begin{lemma}\label{lemma:mollification-Gaussian-process}
	Let $(\xi(\bz))_{\bz\in \T^d}$ be a Gaussian process with continuous paths almost surely.
	For any $k\in \N\cup\{\infty\}$ and $P\in C^k(\T^d)$, define
	\[
	\psi(\bz)\coloneq (P*\xi)(\bz)=\int_{\T^d} P(\bw)\xi(\bz\bar{\bw})\,\boldsymbol{m}(\mathrm{d}\bw),\quad \forall \bz\in \T^d.
	\]
	Then the Gaussian process $(\psi(\bz))_{\bz\in \T^d}$ has $C^k$-paths almost surely and its covariance kernel  is  given by
	\[
	\E[\psi(\bz_1)\psi(\bz_2)]=\iint_{\T^d\times \T^d} P(\bw_1)P(\bw_2)\E\big[\xi(\bz_1\bar{\bw}_2)\xi(\bz_2\bar{\bw}_2)\big]\,\boldsymbol{m}(\mathrm{d}\bw_1)\boldsymbol{m}(\mathrm{d}\bw_2).
	\]
\end{lemma}
\begin{proof}
	It suffices to show that $((P*\xi)(\bz))_{\bz\in \T^d}$ is a Gaussian process. Let $\bz_1,\bz_2,\ldots,\bz_n\in \T^d$ and we shall show that
	\[
	\sum_{k=1}^n \alpha_k (P*\xi)(\bz_k)=\int_{\T^d} P(\bw)\cdot \bigg(\sum_{k=1}^n \alpha_k \xi(\bz_k\bar{\bw} )\bigg)\,\boldsymbol{m}(\mathrm{d}\bw)
	\]
	is still a Gaussian random variable for any $\alpha_1,\alpha_2,\ldots,\alpha_n\in \R$. But now, $(\sum_{k=1}^n \alpha_k \xi(\bz_k\bw))_{\bw\in \T^d}$ forms a new continuous Gaussian process on $\T^d$. Therefore, it is sufficient to prove that $(P*\xi)(\boldsymbol{1})$ is Gaussian.
	
	Let $\mathcal{P}_n$ be a partition of $[-1/2,1/2)^d$ into cubes of side length $1/n$. For each $\boldsymbol{I}\in \mathcal{P}_n$, we arbitrary pick a point $\bw_{\boldsymbol{I}}\in \mathcal{E}(\boldsymbol{I})$. Denote the volume of $\boldsymbol{I}$ in Lebesgue measure by $\operatorname{Vol}(\boldsymbol{I})$. Then the Riemann sum
	\[
	\sum_{\boldsymbol{I}\in \mathcal{P}_n} P(\bw_{\boldsymbol{I}})\xi(\boldsymbol{1}\cdot \bar{\bw}_{\boldsymbol{I}})\cdot \operatorname{Vol}(\boldsymbol{I})=n^{-d}\sum_{\boldsymbol{I}\in \mathcal{P}_n} P(\bw_{\boldsymbol{I}})\xi(\boldsymbol{1}\cdot\bar{\bw}_{\boldsymbol{I}})
	\]
	converges to $(P*\xi)(\boldsymbol{1})$ as $n\to \infty$. Since $\xi(\bz)$ is a Gaussian process, the Riemann sum is a Gaussian random variables. Therefore, the limit random variable $(P*\xi)(\boldsymbol{1})$ is Gaussian.
\end{proof}

\subsubsection{Pisier's martingale type inequality}\label{sssec:Pisier-mar-p-ineq}

Our global estimates rely on a powerful result from the theory of Banach space geometry, Pisier's martingale type inequality, which allows us to control the norm of a vector-valued martingale by the sum of the norms of its increments.

We shall use the following well-known fact in the theory of Banach space geometry (see~\cite[Proposition~10.36 and Definition~10.41]{Pis16}): 
\begin{center}
	{\it  For any  $2\le q<\infty$, the Banach space $\ell^q$ has martingale type $p$ for all $1<p\le 2$. }
\end{center}
More precisely,  for any  $1<p\le 2\le q<\infty$, there exists a constant $C(p, q)>0$ such that any $\ell^q$-vector-valued martingale $(F_m)_{m\ge 0}$ in $L^p(\PP; \ell^q)$ satisfies
\begin{align}\label{eq:def-Mtype}
	\E [\| F_m\|_{\ell^q}^p] \le C(p,q) \sum_{k=0}^m \E[\|F_k- F_{k-1}\|_{\ell^q}^p],
\end{align}
with the convention $F_{-1}\equiv 0$. 	The inequality \eqref{eq:def-Mtype} implies in particular that for any family of independent and  centered  $\ell^q$-valued random variables $(G_k)_{k=0}^m$  in $L^p(\PP; \ell^q)$, 
\begin{align}\label{eq:def-ind-Mtype}
	\E\Big[\Big\|\sum_{k=0}^m G_k\Big\|_{\ell^q}^p\Big] \le C(p,q)\sum_{k =0}^m \E[\|G_k \|_{\ell^q}^p].  
\end{align} 

For further reference, note that when $1<p\le 2$, if $(T_m)_{m\ge 0}$ is a scalar martingale, the inequality \eqref{eq:def-Mtype} follows from the classical Burkholder's martingale inequality and reads as 
\[
\E[|T_m|^p]\le C(p)\sum_{k=0}^m\E[|T_k-T_{k-1}|^p],
\]
with the convention $T_{-1}\equiv 0$.  Moreover,  the inequality \eqref{eq:def-ind-Mtype} implies that for any family of independent and  centered random variables $(Y_k)_{k=0}^m$  in $L^p(\PP)$, 
\begin{align}\label{eq:def-ind-Mtype-scalar}
	\E\Big[\Big|\sum_{k=0}^m Y_k \Big|^p\Big]\le C(p)\sum_{k=0}^m \E[|Y_k|^p].  
\end{align} 

\subsubsection{Smooth partition of unity on $\T^d$}\label{sssec:pou}

Our local analysis requires a smooth localization tool that is adapted to the dyadic scaling. This is achieved by the following smooth partition of unity, whose construction is detailed in Appendix~\ref{appendix:pou}.
\begin{lemma}[Smooth partition of unity on $\T^d$]\label{lemma:pou-scaling-Td}
	Let $d\ge 1$ be an integer. There exists a sequence of $C^\infty(\T^d)$-functions $\{\varphi_k\}_{k\in \N\setminus\{0\}}$ such that for any $k\ge 1$, the following properties hold:  
	\begin{itemize}
		\item[(1)] The translations of $\varphi_k$ form a partition of unity on $\T^d$, that is,
		\begin{equation}\label{eq:pou-Td-euclidean-coordinate}
			\sum_{\boldsymbol{I}\in \mathscr{D}_k} \varphi_{k}(\mathcal{E}(\bt-\boldsymbol{c_I}))=1,\quad \forall \bt\in [-1/2,1/2)^d.
		\end{equation}
		\item[(2)] The support of $\varphi_k$ is contained in $\mathcal{E}([-2^{-k},2^{k}]^d)\coloneq \{\mathcal{E}(\bt):\bt\in [-2^{-k},2^{k}]^d \}$.
		\item[(3)] For each $\boldsymbol{\alpha}\in \N^d$, the derivatives of $\varphi_k$ satisfy the scaling property
		\[
		\|D^{\boldsymbol{\alpha}} \varphi_k\|_{L^\infty(\T^d)}\le C(|\boldsymbol{\alpha}|) 2^{k|\boldsymbol{\alpha}|},
		\]
		where the constant $C(|\boldsymbol{\alpha}|)$ is independent of $k$.
	\end{itemize}
\end{lemma}

Define $\varphi_{\boldsymbol{I}}(\mathcal{E}(\bt))\coloneq\varphi_k(\mathcal{E}(\bt-\boldsymbol{c_I}))$ for each $\boldsymbol{I}\in \mathscr{D}_k$. Then \eqref{eq:pou-Td-euclidean-coordinate} can be rewritten as
\[
\sum_{\boldsymbol{I}\in \mathscr{D}_k} \varphi_{\boldsymbol{I}}(\bt)=1,\quad \forall \bt\in \T^d,
\]
and the support of $\varphi_{\boldsymbol{I}}$
\[
\supp \varphi_{\boldsymbol{I}}\subset \mathcal{E}(2\boldsymbol{I})\coloneq \{\mathcal{E}(\bt):\bt\in 2\boldsymbol{I}\}.
\]
Since 
\[
[D^{\boldsymbol{\alpha}}\varphi_k(\mathcal{E}(\cdot-\bs))](\bt)=(D^{\boldsymbol{\alpha}}\varphi_k)(\mathcal{E}(\bt-\bs)),\quad \forall \bt,\bs\in \T^d,
\]
we have $\|D^{\boldsymbol{\alpha}}\varphi_{\boldsymbol{I}}\|_{L^\infty(\T^d)}=\|D^{\boldsymbol{\alpha}}\varphi\|_{L^\infty(\T^d)}$. Hence the scaling property of the derivatives of $\varphi_{\boldsymbol{I}}$ is also satisfied, that is,
\[
\|D^{\boldsymbol{\alpha}} \varphi_{\boldsymbol{I}}\|_{L^\infty(\T^d)}\le C(|\boldsymbol{\alpha}|) 2^{k|\boldsymbol{\alpha}|}.
\]

\section{Existence of log-correlated Gaussian field with smooth decomposition}\label{sec:smooth-decomposition}

In this section, we construct the sequence of smooth Gaussian processes that forms the foundation of our work. Let $(\psi_j(\bz))_{\bz\in \T^d}$ be a sequence of stationary, centered Gaussian processes, indexed by $j=1, 2, \ldots$. We define their covariance kernel as
\[
  K_j(\bz, \bw) \coloneq \E[\psi_j(\bz) \psi_j(\bw)].
\]
Due to stationarity, this kernel depends only on the product $\bz\bar{\bw}$. For convenience in later arguments, we also define the associated one-variable function
\[
  \mathcal{K}_j(\bz) \coloneq K_j(\bz, \boldsymbol{1}).
\]
The following proposition establishes the existence of a sequence of processes whose covariance kernels satisfy a crucial combination of regularity properties (such as smoothness of paths and uniform derivative bounds) and structural properties (such as finite dependence and a logarithmic covariance sum).
\begin{proposition}\label{prop:smooth-decomposition-log-correlated-field}
There exists a sequence of independent, stationary, centered Gaussian processes, denoted by $(\psi_j(\bz))_{\bz \in \T^d}$ for each $j\in \N\setminus\{0\}$, whose associated covariance kernels $\mathcal{K}_j$ and functions $K_j$ (as defined above) satisfy the following properties:
\begin{enumerate}
    \item[(P0)] (Smoothness). The sample paths of all the Gaussian processes $(\psi_j(\bz))_{\bz\in \T^d}$ belong to $C^\infty(\T^d)$ almost surely.

    \item[(P1)] (Finite dependence). If $d_{\T^d}(\bz, \bw) \ge 3 \cdot 2^{-j}$, then the kernel $K_j(\bz, \bw) = 0$.
    
    \item[(P2)] (Sum to log-correlation). The sum of the covariance kernels is of logarithmic type:
    \[
    \sum_{j=1}^{\infty} K_j(\bz, \bw) = \log_+ \frac{1}{d_{\T^d}(\bz, \bw)} + g(\bz\bar{\bw}),    \quad \text{with $g \in C(\T^d)$.}
    \]

    \item[(P3)] (Limiting variance).  The variances of the processes, given by $\mathcal{K}_j(\mathbf{1})$, have the limit:
    \[
    \lim_{j \to \infty} \mathcal{K}_j(\boldsymbol{1}) =\lim\limits_{j\to \infty}\mathbb{E}[\psi_j(\bt)^2]=\log 2.
    \]
    
    \item[(P4)] (Uniform Derivative Bounds). There exists an increasing function $\Theta: (0, \infty) \to [1, \infty)$ such that for all $p > 0$,
    \begin{align}\label{incr-p-moment}
        \sup_{j} \sup_{|\boldsymbol{\alpha}| \le d} \sup_{\bz\in \T^d} \frac{\E[|D^{\boldsymbol{\alpha}}\psi_j(\bz)|^p]}{j^{2|\boldsymbol{\alpha}|p}2^{j|\boldsymbol{\alpha}|p}} \le \Theta(p).
    \end{align}
\end{enumerate}
\end{proposition}

\begin{remark*}
The increasing condition on  the function $\Theta$ in \eqref{incr-p-moment} will be convenient for us later.   We may only assume that $\Theta(p)<\infty$ for all $p>0$  and drop the increasing  condition.  Indeed, if 
\[
\widetilde{\Theta}(p): = \sup_{j} \sup_{|\boldsymbol{\alpha}| \le d} \sup_{\bz\in \T^d} \frac{\E[|D^{\boldsymbol{\alpha}}\psi_j(\bz)|^p]}{j^{2|\boldsymbol{\alpha}|p}2^{j|\boldsymbol{\alpha}|p}} <\infty \quad \text{for all $p>0$},
\]
then by noting that 
$
p\mapsto (\E[|D^{\boldsymbol{\alpha}}\psi_j(\bz)|^p]/ (j^{2|\boldsymbol{\alpha}|p}2^{j|\boldsymbol{\alpha}|p}))^{1/p}
$ is increasing on $p$,  we may obtain an increasing function $\Theta$ satisfying \eqref{incr-p-moment} by setting 
\[
\Theta(p)=  \sup_{0<q\le p} \widetilde{\Theta}(q)<\infty. 
\]
\end{remark*}

\subsection{Construction of smooth processes in Proposition~\ref{prop:smooth-decomposition-log-correlated-field}}\label{ssec:construc-smooth-decom}
Let $\Phi\colon \R^d\to \R$ be a $C^\infty$ non-negative, positive definite, isotropic function, in the sense that there exists a function $f:[0,\infty)\to \R$ such that $\Phi(\bt)=f(\|\bt\|)$. Furthermore, we also require $\Phi$ to satisfy
\begin{equation}\label{eq:properties-Phi}
\supp \Phi\subset \overline{B}(\boldsymbol{0},1),\quad \Phi(\boldsymbol{0})=1 \anand C_\Phi\coloneq\sup_{\bt\in \R^d\setminus \{\boldsymbol{0}\}} \frac{|\Phi(\bt) - \Phi(\boldsymbol{0})|}{\|\bt\|^{2}}<\infty.
\end{equation}
The existence of such function was confirmed in~\cite[Lemma~3.9]{LQT25}. For any integer $b\ge 2$, define
\begin{equation}\label{eq:j-kernel-Rd}
\mathcal{L}_j(\bt)=\int_1^2 \frac{\Phi(2^j u\bt)}{u}\,\mathrm{d}u,\quad \forall \bt\in \R^d, j\in \N.
\end{equation}
Then $\mathcal{L}_j$ is a positive definite function with $\supp \mathcal{L}_j\subset \overline{B}(\boldsymbol{0},2^{-j})$ for each $j\in \N$. Moreover, by~\cite[Proposition~3.8]{LQT25}, we know that
\begin{equation}\label{eq:property-Lj}
\sum_{j=1}^\infty \mathcal{L}_j(\bt)=\log_+\frac{1}{\|\bt\|}+g_0(\|\bt\|),\quad \forall \bt\in \overline{B}(\boldsymbol{0},\sqrt{d})\setminus\{\boldsymbol{0}\},
\end{equation}
where $g_0\in C^1([0,\sqrt{d}])$. 

We will construct the desired processes in Proposition~\ref{prop:smooth-decomposition-log-correlated-field} within two steps. We first construct independent Gaussian processes $(\xi_j(\bz))_{\bz\in \T^d}$ with the ``kernel function'' $\mathcal{L}_j$ in a proper sense. In the second step, we shall  mollify $\xi_j$ via convolution with a smooth function. 
This mollification step is essential to endow the processes with the $C^\infty$-regularity required by property (P0) of Proposition~\ref{prop:smooth-decomposition-log-correlated-field}. The flowchart of our construction is exhibit in Figure~\ref{fig:construct-smooth-decom}.

\begin{figure}[htpb]
\centering
\begin{tikzpicture}[
    node distance=2.7cm and 3.5cm,
    every node/.style={font=\small, align=center},
    box/.style={draw, rounded corners, minimum width=3.5cm, minimum height=1.2cm, fill=gray!10},
    arrow/.style={-{Latex[length=3mm]}, thick, shorten >=6pt, shorten <=6pt},
	longarrow/.style={-{Latex[length=3mm]}, thick, shorten >=10pt, shorten <=20pt},
    uparrow/.style={-{Latex[length=3mm]}, thick, shorten >=6pt, shorten <=6pt},
    dashedarrow/.style={-{Latex[length=3mm]}, thick, dashed, shorten >=10pt, shorten <=20pt}
  ]

% Step 1 nodes
\node[box] (Lj) {$\mathcal{L}_j$\\p.d. function on $\R^d$};
\node[box, right= 4em of Lj] (Hj) {$\mathcal{H}_j$\\p.d. function on $\T^d$};
\node[box, below= of Hj] (xij) {$\xi_j$\\non-smooth process\\on $\T^d$};

% Step 2 nodes
\node[box, right= 12em of Hj] (Kj) {$\mathcal{K}_j$\\p.d. function on $\T^d$};
\node[box, below=of Kj] (psij) {$\psi_j$\\smooth process\\on $\T^d$};
\node[right= 1.5em of xij] (P1) {};

% arrows
\draw[arrow] (Lj) -- (Hj);
\draw[arrow] (Hj) -- (xij);
\draw[dashedarrow] (Hj) -- node[above, yshift=2pt] {$\mathcal{K}_j=\mathcal{H}_j*P_j*\widetilde{P}_j$} (Kj);
\draw[longarrow] (xij) -- node[below, yshift=-2pt] {convolution with $P_j$} (psij);
\draw[uparrow] (psij) -- (Kj);

% Step 1 box
\node[draw, rounded corners, thick, dashed, fit=(Lj) (Hj) (xij), inner sep=0.2cm, label=below:{\strut \bf Step 1}] {};

% Step 2 box
\node[draw, rounded corners, thick, dashed, fit=(Kj) (psij) (P1), inner xsep=0.2cm, inner ysep=0.2cm, label=below:{\strut \bf Step 2}] {};

\end{tikzpicture}
\caption{Construction of smooth processes in Proposition~\ref{prop:smooth-decomposition-log-correlated-field} (p.d. = positive definite).}
\label{fig:construct-smooth-decom}
\end{figure}
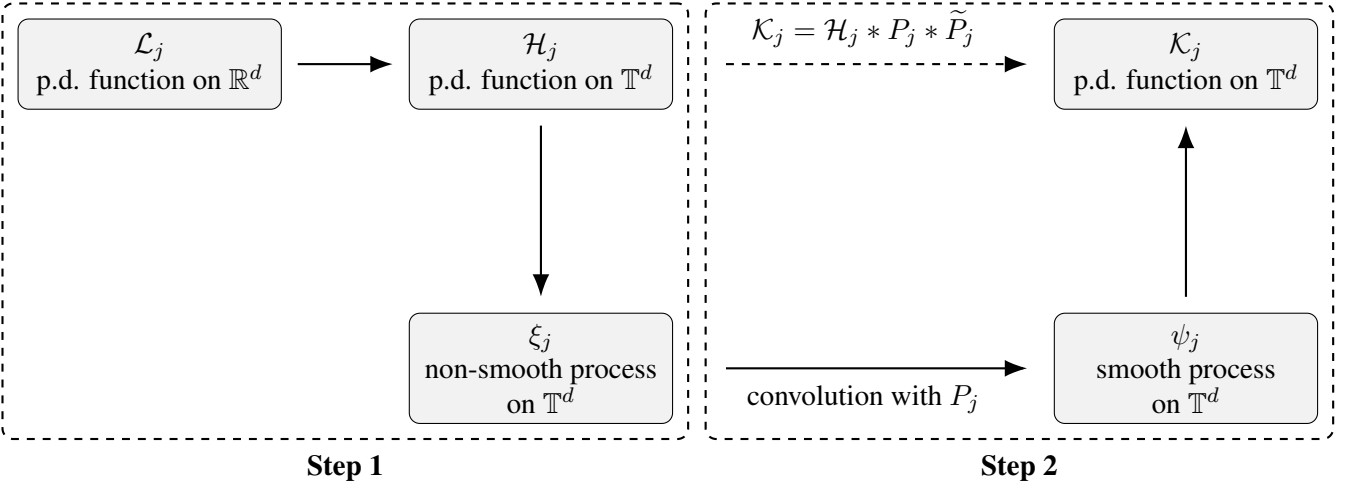

\medskip
{\flushleft \bf Step1. Construct Gaussian processes $(\xi_j(\bz))_{\bz\in \T^d}$.}
\medskip

We begin with constructing kernel functions $\mathcal{H}_j$ on $\T^d$ by $\mathcal{L}_j$. 
\begin{lemma}\label{lemma:pd-function-Rd -Td}
	Let $\mathcal{L}$ be a positive definite function on $\R^d$ with $\supp \mathcal{L}\subset [-1/2,1/2]^d$. Define 
	\[
	\mathcal{H}(\mathcal{E}(\bt))\coloneq \mathcal{L}(\bt),\quad \forall \bt\in [-1/2,1/2)^d.
	\]
	 Then $\mathcal{H}$ is a positive definite function on $\T^d$.
\end{lemma}
\begin{proof}
	We will use $\mathcal{F}_{\T^d}$ and $\mathcal{F}_{\R^d}$ to distinguish Fourier transforms on $\T^d$ and $\R^d$ for the moment. Since $\mathcal{L}$ is positive definite on $\R^d$, we have $(\mathcal{F}_{\R^d}\mathcal{L})(\boldsymbol{\xi})\ge 0$ for all $\boldsymbol{\xi}\in \R^d$.
	
	By definition, for any $\boldsymbol{n}=(n_1,n_2,\ldots,n_d)\in \Z^d$, we have
	\begin{align*}
(\mathcal{F}_{\T^d}\mathcal{H})(\boldsymbol{n})&=\int_{[-1/2,1/2)^d}\mathcal{H}(\mathcal{E}(\bt))e_{\boldsymbol{n}}(\bt)\,\mathrm{d}\bt\\
	&=\int_{[-1/2,1/2)^d}\mathcal{L}(\bt)e_{\boldsymbol{n}}(\bt)\,\mathrm{d}\bt=(\mathcal{F}_{\R^d}\mathcal{L})(\boldsymbol{n})\ge 0.
	\end{align*}
	Therefore, the Bochner's theorem (see, e.g.,~\cite[Section~1.4.3]{Rud62}) implies that $\mathcal{K}$ is a positive definite function on $\T^d$.
\end{proof}

Let $\mathcal{L}_j$ be the function in \eqref{eq:j-kernel-Rd}. Define
\begin{equation}\label{eq:def-Hj}
\mathcal{H}_j(\mathcal{E}(\bt))=\mathcal{L}_j(\bt),\quad \forall \bt\in [-1/2,1/2)^d.
\end{equation}
Then, by Lemma~\ref{lemma:pd-function-Rd -Td}, $\mathcal{H}_j$ is a positive definite function on $\T^d$. 
\begin{lemma}\label{lemma:properties-Hj}
	The positive definite function $\mathcal{H}_j$ satisfies the following properties:
\begin{itemize}
	\item[(1)] For all $\bz\in \T^d$, $\mathcal{H}_j(\bz)=\mathcal{H}_j(\bar{\bz})$.
    \item[(2)] The support of $\mathcal{H}_j$ is contained in $\overline{B}(\boldsymbol{1},2^{-j})$.
    \item[(3)] The sum of $\mathcal{H}_j$ is of log-type, that is,
    \[
	\sum_{j=1}^\infty \mathcal{H}_j(\bz)=\log_+\frac{1}{\|\bz\|}+g_0(\|\bz\|),\quad \forall \bz\in \T^d\setminus\{\boldsymbol{1}\}.
	\]
\end{itemize}
\end{lemma}
\begin{proof}
	(1) For $\bz=(e(t_1),e(t_2),\ldots,e(t_d))\in \T^d$ with $\bt=(t_1,t_2,\ldots,t_d)\in (-1/2,1/2)^d$, we have
	\[
	\mathcal{H}_j(\bz)=\mathcal{L}_j(\bt)=\int_1^2 \frac{\Phi(2^j u\bt)}{u}\,\mathrm{d}u=\int_1^2 \frac{\Phi(2^j u(-\bt))}{u}\,\mathrm{d}u=\mathcal{L}_j(-\bt)=\mathcal{H}_j(\bar{\bz}).
	\]
	The third equality holds since $\Phi$ is isotropic.

	If $\bz=(e(-1/2),e(-1/2),\ldots,e(-1/2))=(-1,-1,\ldots,-1)$, then $\bz=\bar{\bz}$. In this case, $\mathcal{H}_j(\bz)=\mathcal{H}_j(\bar{\bz})$ holds automatically. Therefore, we have shown that $\mathcal{H}_j(\bz)=\mathcal{H}_j(\bar{\bz})$ for all $\bz\in \T^d$.

	(2) For $\bz\in \T^d$ with $d_{\T^d}(\bz,\boldsymbol{1})>2^{-j}$, the canonical representation $(t_1,t_2,\ldots,t_d)$ of $\bz$ in $[-1/2,1/2)$ satisfies
	\[
	\bigg(\sum_{j=1}^d |t_j|^2\bigg)^{1/2}> 2^{-j}.
	\]
	Combining with the fact that $\supp \mathcal{L}_j\subset \overline{B}(\boldsymbol{0},2^{-j})$, we have 
	\[
	\mathcal{H}_j(\bz)=\mathcal{L}_j(t_1,t_2,\ldots,t_d)=0.
	\]
	Therefore, the support of $\mathcal{H}_j$ lies in $\overline{B}_{\T^d}(\boldsymbol{1},2^{-j})$.

	(3) Let $\bz=(e(t_1),e(t_2),\ldots,e(t_d))\in \T^d\setminus\{\boldsymbol{1}\}$ with $\bt=(t_1,t_2,\ldots,t_d)\in [-1/2,1/2)^d$. Then
	\[
	\mathcal{H}_j(\bz)=\mathcal{L}_j(t_1,t_2,\ldots,t_d)=\mathcal{L}_j(\bt).
	\]
	Therefore, by \eqref{eq:property-Lj}, we have
	\[
	\sum_{j=1}^\infty \mathcal{H}_j(\bz)=\sum_{j=1}^\infty \mathcal{L}_j(\bt)=\log_+ \frac{1}{\|\bt\|}+g_0(|\bt|)=\log_+ \frac{1}{\|\bz\|}+g_0(\|\bz\|).
	\]
	This completes the proof.
\end{proof}

\begin{lemma}\label{lemma:Lp-Gaussian-equiv}
	If $\xi$ is a Gaussian random variable, then for any $0<p<\infty$, we have
	\[
	\big(\E[|\xi|^p]\big)^{1/p}=\sqrt{2}\cdot \bigg(\frac{\Gamma((p+1)/2)}{\sqrt{\pi}}\bigg)^{1/p}\cdot \big(\E[|\xi|^2]\big)^{1/2}.
	\]
\end{lemma}

\begin{corollary}\label{corollary:original-process}
	For each $j\in \N\setminus\{0\} $, there exists a continuous centered Gaussian process $(\xi_j(\bz))_{\bz\in \T^d}$ whose kernel function is exactly $\mathcal{H}_j$. 
\end{corollary}
\begin{proof}
	Let $\eta_j$ be a centered Gaussian process with kernel function $\mathcal{H}_j$.  By Kolmogorov's continuity theorem (see, e.g.,~\cite[Chapter~4, Theorem~4.23]{Kal21}), it is sufficient to verify that 
	\[
	\E[|\eta_j(\bz)-\eta_j(\bw)|^{2d}]\le C(j,d,C_\Phi) (d_{\T^d}(\bz,\bw))^{2d},\quad \forall \bz,\bw\in \T^d,
	\]
	where $C(j,d,C_\Phi)$ is a constant depending on $j$, $d$ and $C_\Phi$.

    For $\bz,\bw\in \T^d$, let $\bt=(t_1,t_2,\ldots,t_d)\in [-1/2,1/2)^d$, such that $\bz\bar{\bw}=(e(t_1),e(t_2),\ldots,e(t_d))$. Then
	\begin{align*}
		\E[|\eta_j(\bz)-\eta_j(\bw)|^2]=2\mathcal{H}_j(\boldsymbol{1})-2\mathcal{H}_j(\bz\bar{\bw})=2\mathcal{L}_j(\boldsymbol{0})-2\mathcal{L}_j(\bt).
	\end{align*}
	By the definition of $\mathcal{L}_j$ and \eqref{eq:properties-Phi}, we have
	\[
		\E[|\eta_j(\bz)-\eta_j(\bw)|^2]=2\int_1^2 \frac{\Phi(\boldsymbol{0})-\Phi(2^ju\bt)}{u}\,\mathrm{d}u\le 2C_\Phi 2^{2j}\|\bt\|^2\int_1^2 u\,\mathrm{d}u=3C_\Phi 2^{2j}\cdot (d_{\T^d}(\bz,\bw))^2.
	\]
    Therefore, by Lemma~\ref{lemma:Lp-Gaussian-equiv},
    \begin{align*}
    \E[|\eta_j(\bz)-\eta_j(\bw)|^{2d}]&= 2^d\cdot  \bigg(\frac{\Gamma((2d+1)/2)}{\sqrt{\pi}}\bigg)\cdot \bigg(\E[|\eta_j(\bz)-\eta_j(\bw)|^2]\bigg)^d\\
    &\le C_\Phi^d\cdot 6^d\cdot \bigg(\frac{\Gamma((2d+1)/2)}{\sqrt{\pi}}\bigg)\cdot 2^{2jd}\cdot (d_{\T^d}(\bz,\bw))^{2d}.
    \end{align*}
    This shows that $\eta_j$ admits a continuous version denoted by $\xi_j$.
\end{proof}

We always assume that $\{\xi_j\}_{j\in \N\setminus \{0\}}$ are independent.

\medskip
{\flushleft \bf Step2. Mollify the processes $(\xi_j(\bz))_{\bz\in \T^d}$ and calculation of kernel functions. }
\medskip

Let $Q\in C^\infty(\R^d)$ be a non-negative isotropic function with $\supp Q\subset \overline{B}(\boldsymbol{0},1)$. For each $j\in \N\setminus\{0\}$, define
\[
Q_j(\bt)=\varepsilon_j^{-d}Q\bigg(\frac{\bt}{\varepsilon_j}\bigg),\quad \forall \bt\in \R^d,\quad \text{where}~\varepsilon_j\coloneq j^{-2}2^{-j}.
\]
 The specific decay rate of $\varepsilon_j$ is carefully chosen to satisfy two distinct requirements in the subsequent proofs.
\begin{itemize}
    \item \textit{Exponential component.} The term $2^{-j}$ is chosen to match the natural scale of the $j$-th process. 
    \item \textit{Polynomial component.} The term $j^{-2}$ ensures a sufficiently rapid mollification. This allows the resulting smooth process to closely approximate the original one, a fact that is made precise in Lemma~\ref{lemma:convergence-difference}.
\end{itemize}

The mollifier $P_j \in C^\infty(\T^d)$ is defined by $P_j(\mathcal{E}(\mathbf{t})) \coloneqq Q_j(\bt)$ for all $\bt \in [-1/2, 1/2)^d$. Its support is contained within a small ball centered at the identity $\boldsymbol{1}$. For convenience, we denote this ball by $\mathcal{B}_j$, defined as
\begin{equation}\label{eq:def-mathcalBj}
	\mathcal{B}_j \coloneqq \overline{B}_{\mathbb{T}^d}(\mathbf{1}, j^{-2}2^{-j}).
\end{equation}
Thus, we can write the support condition as
\begin{equation}\label{eq:supp-Pj}
	\supp P_j \subset \mathcal{B}_j.
\end{equation}

Define
\[
\psi_j(\bz)\coloneq (P_j*\xi_j)(\bz)=\int_{\T^d} P_j(\bw)\xi_j(\bz\bar{\bw})\,\boldsymbol{m}(\mathrm{d}\bw),\quad \forall \bz\in \T^d.
\]
We claim that $\psi_j$ is the desirable process in Proposition~\ref{prop:smooth-decomposition-log-correlated-field}. 

We now compute the covariance kernel of $\psi_j$, which we denote by $K_j(\bz_1, \bz_2)$: 
\begin{align*}
K_j(\bz_1,\bz_2)\coloneq \E[\psi_j(\bz_1)\psi_j(\bz_2)]&=\iint_{\T^d\times \T^d} P_j(\bw_1)P_j(\bw_2)\E\big[\xi_j(\bz_1\bar{\bw}_1)\xi_j(\bz_2\bar{\bw}_2)\big]\,\boldsymbol{m}(\mathrm{d}\bw_1)\boldsymbol{m}(\mathrm{d}\bw_2)\\
&=\iint_{\T^d\times \T^d} P_j(\bw_1)P_j(\bw_2)\mathcal{H}_j(\bz_1\bar{\bz}_2\cdot \bar{\bw}_1\bw_2)\,\boldsymbol{m}(\mathrm{d}\bw_1)\boldsymbol{m}(\mathrm{d}\bw_2).
\end{align*}
We adopt the standard notationin in harmonic analysis:
\[
\widetilde{P}_j(\bz)\coloneq P_j(\bar{\bz}),\quad \forall \bz\in \T^d.
\]
Then we can rewrite the covariance function as
\begin{align*}
\mathcal{K}_j(\bz_1,\bz_2)&=\int_{\T^d}\widetilde{P}_j(\bar{\bw}_2)\int_{\T^d}P_j(\bw_1)\mathcal{H}_j(\bz_1\bar{\bz}_2\bw_2\cdot \bar{\bw}_1)\,\boldsymbol{m}(\mathrm{d}\bw_1)\boldsymbol{m}(\mathrm{d}\bw_2)\\
&=\int_{\T^d} \widetilde{P}_j(\bar{\bw}_2)(\mathcal{H}_j*P_j)(\bz_1\bar{\bz}_2\cdot \bw_2)\,\boldsymbol{m}(\mathrm{d}\bw_2)\\
&=\int_{\T^d} \widetilde{P}_j(\bw_2)(\mathcal{H}_j*P_j)(\bz_1\bar{\bz}_2\cdot \bar{\bw}_2)\,\boldsymbol{m}(\mathrm{d}\bw_2)\\
&=((\mathcal{H}_j*P_j)*\widetilde{P}_j)(\bz_1\bar{\bz}_2).
\end{align*}
Therefore, the kernel function of $\psi_j$ is given by
\begin{equation}\label{eq:def-Kj}
\mathcal{K}_j(\bz)\coloneq \mathcal{K}_j(\bz,\boldsymbol{1})=(\mathcal{H}_j*P_j*\widetilde{P}_j)(\bz),\quad \forall \bz\in \T^d.
\end{equation}

\subsection{Proof of Proposition~\ref{prop:smooth-decomposition-log-correlated-field}}

We need to verify the properties (P0)--(P4) in Proposition~\ref{prop:smooth-decomposition-log-correlated-field} for the process $(\psi_j(\bz))_{\bz\in \T^d}$.

\medskip
{\flushleft \bf Verification of (P0).} Since $\{\eta_j\}_{j\in \N\setminus\{0\}}$ are required to be independent, it follows that the mollified processes $\{\psi_j\}_{j\in \N\setminus\{0\}}$ are independent as well. By Lemma~\ref{lemma:mollification-Gaussian-process}, we know that the process $\psi_j(\bz)$ has $C^\infty$-path almost surely for each $j\in \N\setminus\{0\}$. Hence (P0) in Proposition~\ref{prop:smooth-decomposition-log-correlated-field} is satisfied.

\medskip
{\flushleft \bf Verification of (P1).}
To verify (P1), we should analyze the support of the function $\mathcal{K}_j$, which was derived in \S\ref{ssec:construc-smooth-decom} as $\mathcal{K}_{j}=\mathcal{H}_{j}*(P_{j}*\widetilde{P}_{j})$.

\begin{lemma}
Let $f,g\in L^2(\T^d)$. Then $\supp(f*g)\subset (\supp f)\cdot (\supp g)$, where
\[
(\supp f)\cdot (\supp g)\coloneq \{\bw_1\cdot \bw_2:\bw_1\in \supp f,\bw_2\in \supp g\}.
\]
\end{lemma}

By Lemma~\ref{lemma:properties-Hj}, we have $\supp \mathcal{H}_j\subset \overline{B}(\boldsymbol{1},2^{-j})$. Since $\supp P_j\subset \mathcal{B}_j$, where $\mathcal{B}_j=\overline{B}_{\T^d}(\boldsymbol{1},j^{-2}2^{-j})$ as we define in \eqref{eq:def-mathcalBj}, we have
\[
\supp \mathcal{K}_j=\supp (\mathcal{H}_j*P_j*\widetilde{P}_j)\subset \overline{B}_{\T^d}(\boldsymbol{1},2^{-j})\cdot \overline{B}_{\T^d}(\boldsymbol{1},j^{-2}2^{-j})\cdot \overline{B}_{\T^d}(\boldsymbol{1},j^{-2}2^{-j}).
\]
By the triangle inequality, we have
\[
\overline{B}_{\T^d}(\boldsymbol{1},2^{-j})\cdot \overline{B}_{\T^d}(\boldsymbol{1},j^{-2}2^{-j})\cdot \overline{B}_{\T^d}(\boldsymbol{1},j^{-2}2^{-j})\subset \overline{B}_{\T^d}(\boldsymbol{1},(1+2j^{-2})2^{-j})\subset \overline{B}_{\T^d}(\boldsymbol{1},3\cdot 2^{-j}).
\]
Hence $\supp \mathcal{K}_j\subset \overline{B}_{\T^d}(\boldsymbol{1},3\cdot2^{-j})$ for all $j\in \N\setminus\{0\}$. Note that
\[
\mathcal{K}_j(\bz,\bw)=\mathcal{K}_j(\bz\bar{\bw}),\quad \forall \bz,\bw\in \T^d.
\]
The continuity of $\mathcal{K}_j$ implies that, if $d_{\T^d}(\bz,\bw)\ge 3\cdot2^{-j}$, then $\mathcal{K}_j(\bz,\bw)=0$. This proves (P1).

\medskip
{\flushleft \bf Verification of (P2). }
The following lemma is the key to verifying (P2) and (P3), as it establishes a summable error bound between the kernels $\mathcal{K}_j$ and $\mathcal{H}_j$, ensuring the properties of their sum are preserved.
\begin{lemma}\label{lemma:convergence-difference}
	Let $\mathcal{H}_j$ and $\mathcal{K}_j$ be the positive definite functions defined in \eqref{eq:def-Hj} and \eqref{eq:def-Kj}.
	Then we have
	 \[
	 \sum_{j=1}^\infty \|\mathcal{K}_j-\mathcal{H}_j\|_{L^\infty(\T^d)}<\infty.
	 \]
\end{lemma}
\begin{lemma}\label{lemma:Lipschitz-continuity-kernel}
    Let $\mathcal{H}_j$ be the kernel function in \eqref{eq:def-Hj}. Then
	\[
	|\mathcal{H}_j(\bz)-\mathcal{H}_j(\bw)|\le C 2^jd_{\T^d}(\bz,\bw),\quad \forall \bz,\bw\in \T^d,
	\]
	where $C$ is a constant independent of $j$.
\end{lemma}
\begin{proof}
	Let $\xi_j$ be the Gaussian process in Corollary~\ref{corollary:original-process} and $\bz,\bw\in \T^d$. Then, 
	\[
	|\mathcal{H}_j(\bz)-\mathcal{H}_j(\bw)|=|\E[\xi_j(\bz)\xi_j(\boldsymbol{1})]-\E[\xi_j(\bw)\xi_j(\boldsymbol{1})]|=
	|\E[\xi_j(\boldsymbol{1})(\xi_j(\bz)-\xi_j(\bw))]|.
	\]
	By Cauchy-Schwarz inequality, we have
	\[
	|\mathcal{H}_j(\bz)-\mathcal{H}_j(\bw)|\le (\E[|\xi_j(\boldsymbol{1})|^2])^{1/2}\cdot (\E[|\xi_j(\bz)-\xi_j(\bw)|^2])^{1/2}.
	\]
	Following the same line in the proof of Proposition~\ref{corollary:original-process}, we have
	\[
	\E[|\xi_j(\bz)-\xi_j(\bw)|^2]\le 3C_\Phi 2^{2j}\cdot (d_{\T^d}(\bz,\bw))^2.
	\]
	Therefore,
	\[
	|\mathcal{H}_j(\bz)-\mathcal{H}_j(\bw)|\le (\log 2)^{1/2}\cdot (3C_\Phi )^{1/2}\cdot 2^j d_{\T^d}(\bz,\bw).
	\]
	This proves the lemma.
\end{proof}

\begin{proof}[Proof of Lemma~\ref{lemma:convergence-difference}]
	Fix $\bz=(e(t_1),e(t_2),\ldots,e(t_d))\in \T^d$. By definition,
	\begin{align*}
	\mathcal{K}_j(\bz)=(\mathcal{H}_j*P_j*\widetilde{P_j})(\bz)=\iint_{\T^d\times \T^d} P(\bw_1)P_j(\bar{\bw}_2)\mathcal{H}_j(\bz\cdot\overline{\bw_1\bw_2})\,\boldsymbol{m}(\mathrm{d}\bw_1)\boldsymbol{m}(\mathrm{d}\bw_2).
	\end{align*}
    Since $\supp P_j\subset \mathcal{B}_j$, 
    \begin{align*} 
    \big|\mathcal{K}_j(\bz)-\mathcal{H}_j(\bz)\big|&=\bigg|\iint_{\T^d\times \T^d} P(\bw_1)P_j(\bar{\bw}_2)(\mathcal{H}_j(\bz\cdot\overline{\bw_1\bw_2})-\mathcal{H}_j(\bz))\,\boldsymbol{m}(\mathrm{d}\bw_1)\boldsymbol{m}(\mathrm{d}\bw_2)\bigg|\\
    &\le \iint_{\mathcal{B}_j\times \mathcal{B}_j}P_j(\bw_1)P_j(\bar{\bw}_2)|\mathcal{H}_j(\bz\cdot\overline{\bw_1\bw_2})-\mathcal{H}_j(\bz)|\,\boldsymbol{m}(\mathrm{d}\bw_1)\boldsymbol{m}(\mathrm{d}\bw_2).
    \end{align*}
    By Lemma~\ref{lemma:Lipschitz-continuity-kernel}, we have
    \begin{align*}
   |\mathcal{H}_j(\bz\cdot \overline{\bw_1\bw_2})-\mathcal{H}_j(\bz)|&\le C2^j d_{\T^d}(\bz\cdot \overline{\bw_1\bw_2},\bz)\\
   &=C2^jd_{\T^d}(\bw_1\bw_2,\boldsymbol{1})
   \\
   & \le C2^j\cdot (2j^{-2}2^{-j})=2Cj^{-2}.
	\end{align*}
    Therefore,
	\begin{align*}
	|\mathcal{K}_j(\bz)-\mathcal{H}_j(\bz)|&\le 2Cj^{-2}\cdot \iint_{\mathcal{B}_j\times \mathcal{B}_j} P_j(\bw_1)P_j(\bar{\bw}_2)\,\boldsymbol{m}(\mathrm{d}\bw_1)\boldsymbol{m}(\mathrm{d}\bw_2)\\
	&=2Cj^{-2}\cdot \iint_{[-1/2,1/2)^d\times [-1/2,1/2)^d} Q_j(\boldsymbol{x})Q_j(-\boldsymbol{y})\,\mathrm{d}\boldsymbol{x}\mathrm{d}\boldsymbol{y}\\
	&=2Cj^{-2}\cdot \bigg(\int_{\R^d} Q_j(\boldsymbol{x})\,\mathrm{d}\boldsymbol{x}\bigg)^2=2Cj^{-2}.
	\end{align*}
   Since the choice of $\bz$ is arbitrary, we have
    \[
    \sum_{j=1}^\infty \|\mathcal{K}_j-\mathcal{H}_j\|_{L^\infty(\T^d)}\le 2C\cdot \sum_{j=1}^\infty j^{-2}<\infty.
    \]
    This completes the proof of Lemma~\ref{lemma:convergence-difference}.
\end{proof}
Define
\[
g_1(\bz)\coloneq \sum_{j=1}^\infty (\mathcal{K}_j-\mathcal{H}_j)(\bz),\quad \forall \bz\in \T^d.
\]
Then, by Lemma~\ref{lemma:convergence-difference}, the function $g_1$ is continuous on $\T^d$. By Lemma~\ref{lemma:properties-Hj}, we have
\[
\sum_{j=1}^\infty \mathcal{H}_j(\bz)=\log \frac{1}{\|\bz\|_{\T^d}}+g_0(\|\bz\|_{\T^d}),\quad \forall \bz\in \T^d\setminus\{\boldsymbol{1}\},
\]
where $g_0\in C^1([0,\sqrt{d}])$. Let $g(\bz)\coloneq g_1(\bz)+g_0(\|\bz\|_{\T^d})$, then $g$ is continuous on $\T^d$. Thus, for all $\bz\in \T^d\setminus\{\boldsymbol{1}\}$, we have
\begin{align*}
\sum_{j=1}^\infty \mathcal{K}_j(\bz)&=\sum_{j=1}^\infty (\mathcal{K}_j-\mathcal{H}_j)(\bz)+\sum_{j=1}^\infty \mathcal{H}_j(\bz)\\
&=g_1(\bz)+\log_+ \frac{1}{\|\bz\|_{\T^d}}+g_0(\|\bz\|_{\T^d})\\
&=\log_+ \frac{1}{\|\bz\|_{\T^d}}+g(\bz).
\end{align*}
This completes the verification of (P2).

\medskip
{\flushleft \bf Verification of (P3). }
By definition of $\mathcal{H}_j$ and \eqref{eq:properties-Phi}, we have
\[
\mathcal{H}_j(\boldsymbol{1})=\mathcal{L}_j(\boldsymbol{0})=\int_1^2 \frac{\Phi(\boldsymbol{0})}{u}\,\mathrm{d}u=\log 2.
\]
Since $\lim_{j\to \infty}\|\mathcal{K}_j-\mathcal{H}_j\|_{L^\infty(\T^d)}=0$, we obtain the convergence at $\boldsymbol{1}$, that is,
\[
\lim_{j\to \infty}\mathcal{K}_j(\boldsymbol{1})=\lim_{j\to \infty}\mathcal{H}_j(\boldsymbol{1})=\log 2.
\]

\medskip
{\flushleft \bf Verification of (P4). }
 Inspired by \eqref{eq:diff-convolution} in Lemma~\ref{lemma:convolution-deterministic}, it is necessary to control the $L^\infty$-norm of derivatives of $P_j$.
\begin{lemma}
	For any multi-index $\boldsymbol{\alpha}$ with $|\boldsymbol{\alpha}|\le d$ and $j\ge 1$, we have
	\[
	\|D^{\boldsymbol{\alpha}} P_j(\bz)\|_{L^\infty(\T^d)}\le j^{2(|\boldsymbol{\alpha}|+d)}2^{j(|\boldsymbol{\alpha}|+d)}\|Q\|_{C^d(\R^d)}.
	\]
\end{lemma}
\begin{proof}
	By the definition of $P_j$, we have
	\[
	D^{\boldsymbol{\alpha}} P_j(\bz)=j^{2d}2^{jd}D^{\boldsymbol{\alpha}} Q_j(\bt),
	\]
	where $\bt=(t_1,t_2,\ldots,t_d)$ is the canonical representation of $\bz$ in $[-1/2,1/2)^d$ and $\varepsilon_j=j^{-2}2^{-j}$. By the chain rule, we have
	\[
	\bigg(\partial^{\boldsymbol{\alpha}} Q\bigg(\frac{\cdot}{\varepsilon_j}\bigg)\bigg)(\bt)=\varepsilon_j^{-|\boldsymbol{\alpha}|}(\partial^{\boldsymbol{\alpha}} Q)\bigg(\frac{\bt}{\varepsilon_j}\bigg)=j^{2|\boldsymbol{\alpha}|}2^{j|\boldsymbol{\alpha}|}(\partial^{\boldsymbol{\alpha}} Q)\bigg(\frac{\bt}{\varepsilon_j}\bigg).
	\]
	Therefore, 
	\[
	\|D^{\boldsymbol{\alpha}} P_j(\bz)\|_{L^\infty(\T^d)}=j^{2(|\boldsymbol{\alpha}|+d)}2^{j(\boldsymbol{\alpha}|+d)}\|\partial^{\boldsymbol{\alpha}} Q\|_{L^\infty(\R^d)}\le j^{2(|\boldsymbol{\alpha}|+d)}2^{j(|\boldsymbol{\alpha}|+d)}\cdot\|Q\|_{C^d(\R^d)}.
	\]
	This proves the lemma.
\end{proof}

By Lemma~\ref{lemma:mollification-Gaussian-process}, we know
\[
(D^{\boldsymbol{\alpha}} \psi_j)(\bz)=D^{\boldsymbol{\alpha}} (P_j*\xi_j)(\bz)=((D^{\boldsymbol{\alpha}} P_j)*\xi_j)(\bz).
\]
Recalling that $\supp P_j\subset \mathcal{B}_j$ in \eqref{eq:supp-Pj}, we have
\begin{align*}
|D^{\boldsymbol{\alpha}} \psi_j(\bz)|&\le \int_{\mathcal{B}_j}\|D^{\boldsymbol{\alpha}} P_j\|_{L^\infty(\T^d)}\cdot |\xi_j(\bz)|\,\boldsymbol{m}(\mathrm{d}\bz)\\
&\le \|Q\|_{C^d(\R^d)}\cdot j^{2(|\boldsymbol{\alpha}|+d)}2^{j(|\boldsymbol{\alpha}|+d)}\cdot \int_{\mathcal{B}_j}|\xi_j(\bz)|\,\boldsymbol{m}(\mathrm{d}\bz).
\end{align*}
Let $0<p<\infty$. The $p$-moment of $D^{\boldsymbol{\alpha}}\psi_j$ can be estimated by
\[
\E[|D^{\boldsymbol{\alpha}} \psi_j(\bz)|^p]\le \|Q\|_{C^d(\R^d)}^p\cdot j^{2(|\boldsymbol{\alpha}|+d)p}2^{j(|\boldsymbol{\alpha}|+d)p}\cdot \E\bigg[\bigg|\int_{\mathcal{B}_j}|\xi_j(\bz)|\,\boldsymbol{m}(\mathrm{d}\bz)\bigg|^p\bigg].
\]
By Minkowski's inequality and Lemma~\ref{lemma:Lp-Gaussian-equiv}, we have	
\begin{align*}
\E\bigg[\bigg|\int_{\mathcal{B}_j}|\xi_j(\bz)|\,\boldsymbol{m}(\mathrm{d}\bz)\bigg|^p\bigg]&\le \bigg\{\int_{\mathcal{B}_j}\big(\E[|\xi_j(\bz)|^p]\big)^{1/p}\,\boldsymbol{m}(\mathrm{d}\bz)\bigg\}^p\\
&=2^{p/2}\bigg(\frac{\Gamma((p+1)/2)}{\sqrt{\pi}}\bigg)\cdot \bigg\{\int_{\mathcal{B}_j}\big(\E[|\xi_j(\bz)|^2]\big)^{1/2}\,\boldsymbol{m}(\mathrm{d}\bz)\bigg\}^p\\
&=2^{p/2}\bigg(\frac{\Gamma((p+1)/2)}{\sqrt{\pi}}\bigg)\cdot \bigg\{\int_{\mathcal{B}_j}\big(\mathcal{H}_j(\boldsymbol{1})\big)^{1/2}\,\boldsymbol{m}(\mathrm{d}\bz)\bigg\}^p\\
&=(2\log 2)^{p/2}\bigg(\frac{\Gamma((p+1)/2)}{\sqrt{\pi}}\bigg)\cdot (\Omega_d j^{-2d}2^{-jd})^p,
\end{align*}
where $\Omega_d$ is the volume of unit ball in $\R^d$.

Denote
\[
\Theta(p)=\max\bigg\{1,\|Q\|_{C^d(\R^d)}^p\cdot (2\log 2)^{p/2}\cdot \bigg(\frac{\Gamma((p+1)/2)}{\sqrt{\pi}}\bigg)\cdot \Omega_d^p\bigg\}.
\]
Then for all $0<p<\infty$, we have
\[
\sup_{j}\sup_{|\boldsymbol{\alpha}|\le d}\sup_{z\in \T^d}\frac{\E\big[|D^{\boldsymbol{\alpha}} \psi_j(\bz)|^p\big]}{j^{2|\boldsymbol{\alpha}|p}2^{j|\boldsymbol{\alpha}|p}}\le \Theta(p).
\]
We claim that $\Theta$ is increasing in $p$. Let $f(p)=\big(\frac{\Gamma((p+1)/2)}{\sqrt{\pi}}\big)^{1/p}$. By Lemma~\ref{lemma:Lp-Gaussian-equiv} and Lyapunov's inequality, we can see that $f(p)$ is increasing in $p$. It follows that
\[
\Theta(p)=\max\Big\{1,\big(\|Q\|_{C^d(\R^d)}\cdot (2\log2)^{1/2}\cdot \Omega_d\cdot f(p)\big)^p\Big\}
\]
is increasing in $p$ as well. Hence the constructed processes satisfy (P4).

\section{Fourier decay of Gaussian multiplicative chaos on torus}\label{sec:proof-main-thm}
This section is devoted to proving our main result, Theorem~\ref{thm:Fourier-decay-GMC}. The key to the proof is establishing a sufficiently sharp lower bound for the Fourier dimension of the GMC measure $\mu_\infty$. This is a direct result of Proposition~\ref{prop:uniform-boundedness-Fourier-Lebesgue-norm}. 

Given that the GMC measure is constructed via a scale-by-scale, multiplicative process, a standard, global Fourier analysis is inadequate to capture its fine, intrinsic structure. Therefore, our proof strategy relies on a \textit{time-frequency analysis} that is tailored to this multi-scale nature of GMC.

The core of this method is a \textit{multi-resolution strategy}, which combines scale-dependent geometric localization with fixed frequency analysis.

\begin{itemize}
\item \textit{Spatial localization}: We employ a smooth partition of unity (Lemma~\ref{lemma:pou-scaling-Td}) to create scale-adapted window functions, with support of size $\sim 2^{-k}$. This allows the ``resolution'' of our analysis to match the scale of the measure’s component
being investigated.
\item \textit{Frequency analysis}: On each localized piece, we probe its frequency content by taking inner products with the Fourier basis elements $\{\bz^{\boldsymbol{n}}\}_{\boldsymbol{n}\in \Z^d}$.
\end{itemize}

This intuitive, multi-resolution strategy is formally encapsulated in the local estimate in Proposition~\ref{prop:localization} and the global moment bound in Proposition~\ref{prop:uniform-boundedness-Fourier-Lebesgue-norm}. 

Now, we show how to derive Theorem~\ref{thm:Fourier-decay-GMC} by Proposition~\ref{prop:uniform-boundedness-Fourier-Lebesgue-norm} and Proposition~\ref{prop:smooth-decomposition-log-correlated-field}.

\begin{lemma}\label{lemma:upper-bound}
	Let $d\ge 1$ be an integer. Assume that $K$ is a Gaussian kernel on $\T^d$ of the form:
	\[
	K(\bz,\bw)=\log_+\frac{1}{d_{\T^d}(\bz,\bw)}+g(\bz,\bw),\quad \bz,\bw\in \T^d,
	\]
	where $g$ is bounded continuous on $\T^d\times \T^d$. Then for each $\gamma\in (0,\sqrt{2d})$, almost surely, we have
	\[
	\dim_F (\GMC_K^\gamma)\le D_{\gamma,d}.
	\]
\end{lemma}
Lemma~\ref{lemma:upper-bound} states a well-known upper bound for the Fourier dimension of GMC, see for instance Bertacco \cite[Theorem~3.1 and Formula~(3.2)]{Ber23}, Rhodes--Vargas \cite[Section~4.2]{RV14}, and Garban--Vargas \cite[Remark~2]{GV23}. The clear formulation we use here is a direct adaptation of \cite[Lemma~3.6]{LQT25} which was stated for the setting of $[0,1]^d$.

\begin{proof}[Proof of Theorem~\ref{thm:Fourier-decay-GMC}]
By (P3) in Proposition~\ref{prop:smooth-decomposition-log-correlated-field}, the remainder $g$ of kernel function $K$ is a bounded continuous function. Hence we can apply Lemma~\ref{lemma:upper-bound} to obtain the upper bound of the Fourier dimension of $\mu_\infty$, which states that 
\[
\operatorname{dim}_F(\mu_\infty)\le D_{\gamma,d},\quad \text{a.s.},
\]
for $\gamma\in (0,\sqrt{2d})$. 

By Proposition~\ref{prop:uniform-boundedness-Fourier-Lebesgue-norm} and the vector-valued martingale theory, we can obtain
\[
\E\big[\|\mu_\infty\|_{\mathcal{F}L^{\tau/2,q}}^p\big]=\sup_{m\ge 1}\E\big[\|\mu_m\|_{\mathcal{F}L^{\tau/2,q}}^p\big]<\infty
\]
for any $\tau\in (0,D_{\gamma,d})$. Hence, we have, almost surely,
\[
|\widehat{\mu_\infty}(\boldsymbol{n})|^2=O(|\boldsymbol{n}|^{-\tau}),\quad \text{as}~|\boldsymbol{n}|\to \infty.
\] 
Therefore, the GMC measure $\mu_\infty$ satisfies
\[
\operatorname{dim}_F(\mu_\infty)\ge D_{\gamma,d},\quad \text{a.s.}
\]
This completes the proof of Theorem~\ref{thm:Fourier-decay-GMC}.
\end{proof}

The remainder of this section is dedicated to the proof of the crucial Proposition~\ref{prop:uniform-boundedness-Fourier-Lebesgue-norm}.

\subsection{Proof of Proposition~\ref{prop:uniform-boundedness-Fourier-Lebesgue-norm}}\label{ssec:proof-uniform-bdn}

The natural filtration $(\mathscr{G}_{m})_{m\ge 1}$ is defined by
\[
\mathscr{G}_m\coloneq \sigma\{X_1,X_2,\ldots,X_m\},
\]
that is, the $\sigma$-algebra generated by the random processes $X_1,X_2,\ldots,X_m$ introduced in \eqref{eq:def-Xj}.
\medskip
{\flushleft \bf Step 1. Decoupling of different scales.}
\medskip

\begin{lemma}
The sequence of random vectors $\{(\langle \boldsymbol{n}\rangle^{\tau/2}\widehat{\mu}_m(\boldsymbol{n}))_{\boldsymbol{n}\in \Z^d}:m\ge 1\}$ is an $\ell^q$-vector-valued martingale.
\end{lemma}
\begin{proof}

The integrability condition $\E[\|\mathcal{M}_m\|_{l^q}] < \infty$ holds for any finite $m$ due to the $C^\infty$-smoothness of the sample paths of $\mu_m$ (a consequence of (P0) in Proposition~\ref{prop:smooth-decomposition-log-correlated-field}).

For each $m\ge 2$ and $\boldsymbol{n}\in \Z^d$, we have
\begin{align*}
\E[\widehat{\mu}_m(\boldsymbol{n})-\widehat{\mu}_{m-1}(\boldsymbol{n})|\mathscr{G}_{m-1}]=\int_{\T^d}\E\Big[\big[\prod_{j=1}^m X_j(\bz)-\prod_{j=1}^{m-1}X_j(\bz)\big]\Big| \mathscr{G}_{m-1}\Big]\bz^{\boldsymbol{n}}\,\boldsymbol{m}(\mathrm{d}\bz).
\end{align*}
Since $X_m$ is independent of $\mathscr{G}_{m-1}$ and $X_j$ is $\mathscr{G}_{m-1}$-measurable for each $1\le j\le m-1$, we have
\begin{align*}
\E\Big[\big[\prod_{j=1}^m X_j(\bz)-\prod_{j=1}^{m-1}X_j(\bz)\big]\Big| \mathscr{G}_{m-1}\Big]&=\big(\E[X_m|\mathscr{G}_{m-1}]-1\big)\cdot \prod_{j=1}^{m-1}X_j(\bz)\\
&=\big(\E[X_m]-1\big)\cdot \prod_{j=1}^{m-1}X_j(\bz)=0.
\end{align*}
Therefore, we have
\[
\E[\widehat{\mu}_m(\boldsymbol{n})-\widehat{\mu}_{m-1}(\boldsymbol{n})|\mathscr{G}_{m-1}]=0,
\]
which implies that $\{(\langle \boldsymbol{n}\rangle^{\tau/2}\widehat{\mu}_m(\boldsymbol{n}))_{\boldsymbol{n}\in \Z^d}:m\ge 1\}$ is an $\ell^q$-vector-valued martingale with respect to the filtration $\{\mathscr{G}_m\}_{m\ge 1}$.
\end{proof}
Denote the $\ell^q$-vector-valued martingale $\{(\langle \boldsymbol{n}\rangle^{\tau/2}\widehat{\mu}_m(\boldsymbol{n}))_{\boldsymbol{n}\in \Z^d}:m\ge 1\}$ by $(\mathcal{M}_m)_{m\ge 1}$. Then we apply \eqref{eq:def-Mtype} on $\E[\|\mathcal{M}_m\|_{\ell^q}^p]$ to decouple components with different scales of $\mu_m$ by
\begin{equation}\label{eq:decoupling-martingale}
\E[\|\mathcal{M}_m\|_{\ell^q}^p]\lesssim \sum_{k=1}^m \E[\|\mathcal{M}_{k}-\mathcal{M}_{k-1}\|_{\ell^q}^p].
\end{equation}

\medskip
{\flushleft \bf Step 2. Decoupling the spatially localized components.}
\medskip

For each $k\ge 1$ and $\boldsymbol{I}\in \mathscr{D}_k$, define
\[
\mathfrak{D}_{\boldsymbol{I}}^\varphi(\boldsymbol{n})\coloneq \langle \boldsymbol{n}\rangle^{\tau/2}\int_{\T^d}\varphi_{\boldsymbol{I}}(\bz)\bigg[\prod_{j=1}^{k-1}X_j(\bz)\bigg]\mathring{X}_k(\bz)\cdot \bz^{\boldsymbol{n}}\,\boldsymbol{m}(\mathrm{d}\bz),
\]
where $\boldsymbol{n}\in \Z^d\setminus\{\boldsymbol{0}\}$ and 
\[
\mathring{X}_k(\bz)=X_k(\bz)-\E[X_k(\bz)]=X_k(\bz)-1.
\]

Recall that 
\[
\mathcal{M}_k-\mathcal{M}_{k-1}=\bigg(\langle \boldsymbol{n}\rangle^{\tau/2}\int_{\T^d}\bigg[\prod_{j=1}^{k-1}X_j(\bz)\bigg]\mathring{X}_k(\bz)\bz^{\boldsymbol{n}}\,\boldsymbol{m}(\mathrm{d}\bz)\bigg)_{\boldsymbol{n}\in \Z^d}.
\]
We insert the partition of unity $\{\varphi_{\boldsymbol{I}}\}_{\boldsymbol{I}\in \mathscr{D}_k}$ into the integral above to obtain
\begin{align*}
\langle \boldsymbol{n}\rangle^{\tau/2}\int_{\T^d}\bigg[\prod_{j=1}^{k-1}X_j(\bz)\bigg]\mathring{X}_k(\bz)\bz^{\boldsymbol{n}}\,\boldsymbol{m}(\mathrm{d}\bz)&=\langle \boldsymbol{n}\rangle^{\tau/2}\int_{\T^d}\bigg[\prod_{j=1}^{k-1}X_j(\bz)\bigg]\mathring{X}_k(\bz)\Big(\sum_{\boldsymbol{I}\in \mathscr{D}_k}\varphi_{\boldsymbol{I}}(\bz)\Big)\bz^{\boldsymbol{n}}\,\boldsymbol{m}(\mathrm{d}\bz)\\
&=\sum_{\boldsymbol{I}\in \mathscr{D}_k}\langle \boldsymbol{n}\rangle^{\tau/2}\int_{\T^d}\bigg[\prod_{j=1}^{k-1}X_j(\bz)\bigg]\mathring{X}_k(\bz)\varphi_{\boldsymbol{I}}(\bz)\bz^{\boldsymbol{n}}\,\boldsymbol{m}(\mathrm{d}\bz)\\
&=\sum_{\boldsymbol{I}\in \mathscr{D}_k}\mathfrak{D}_{\boldsymbol{I}}^\varphi(\boldsymbol{n}).
\end{align*}

It turns out that we can divide $\sum_{\boldsymbol{I}\in \mathscr{D}_k}\mathfrak{D}_{\boldsymbol{I}}^\varphi(\boldsymbol{n})$ into finite parts. The summands in each part are conditionally independent, which allows us to apply \eqref{eq:def-ind-Mtype} to obtain the estimate of $\E[\|\mathcal{M}_k-\mathcal{M}_{k-1}\|_{\ell^q}^p]$.

Now we make the assertion above clear. Define an equivalence relation on $\mathscr{D}_k$ by
\[
\boldsymbol{I}\sim \boldsymbol{I}'\quad \text{if and only if}\quad \boldsymbol{c}_{\boldsymbol{I}}-\boldsymbol{c}_{\boldsymbol{I}'}\in (5\cdot 2^{-k})\Z^d.
\]
Thereby, we can partition $\mathscr{D}_k$ into disjoint subfamilies $\{\mathscr{D}_{k,s}\}_{s=1}^{S_k}$. These unordered subfamilies can be uniquely determined by the following two features.
\begin{itemize}
	\item[(1)] For any $\boldsymbol{I}\in \mathscr{D}_{k,s}$, we have
	\[
	\mathscr{D}_{k,s}=\{\boldsymbol{I}'\in \mathscr{D}_k:\boldsymbol{I}'\sim \boldsymbol{I}\}.
	\]
	\item[(2)] The family of $k$-th generation dyadic cubes $\mathscr{D}_k$ can be decomposed as
	\[
	\mathscr{D}_k=\bigsqcup_{s=1}^{S_k}\mathscr{D}_{k,s}.
	\]
\end{itemize}
\begin{lemma} Let $\mathscr{D}_{k,s}$ be the $s$-th subfamily of $\mathscr{D}_k$ defined above. Then the following properties hold:
\begin{itemize}
    \item[(1)]  For any $s\in \{1,2,\ldots,S_k\}$ and different $\boldsymbol{I},\boldsymbol{I}'\in \mathscr{D}_{k,s}$, $\{X_k(\bz)\}_{\bz\in 2\boldsymbol{I}}$ and $\{X_k(\bz)\}_{\bz\in 2\boldsymbol{I}'}$ are independent. 
    \item[(2)]  The number of partitions $S_k$ can be bounded by a constant $S(d)\coloneq 5^d$.
\end{itemize}
\end{lemma}
\begin{proof}
(1). Let $\boldsymbol{I},\boldsymbol{I}'\in \mathscr{D}_{k,s}$ and $\boldsymbol{I}\ne \boldsymbol{I}'$. Then for any $\bz\in 2\boldsymbol{I}$, $\bz'\in 2\boldsymbol{I}'$, we have $d_{\T^d}(\bz,\bz')\ge 3\cdot 2^{-k}$.
By (P1) in Proposition~\ref{prop:smooth-decomposition-log-correlated-field}, we know that if $d_{\T^d}(\bz,\bz')\ge 3\cdot 2^{-k}$, then 
\[
\E[X_k(\bz)X_k(\bz')]=K_k(\bz\cdot \overline{\bz'})=0,
\]
that is, $X_k(\bz)$ and $X_k(\bz')$ are independent. This proves (1). 

(2). Fix a cube $\boldsymbol{I}\in \mathscr{D}_{k,1}$, whose center $\boldsymbol{c}_{\boldsymbol{I}}=(c_1,c_2,\ldots,c_d)$. Then for any $\mathscr{D}_{k,s}$, we can pick a unique cube $\boldsymbol{I}_s\in \mathscr{D}_{k,s}$, such that the center of $\boldsymbol{c}_{\boldsymbol{I}_s}=(c_{s,1},c_{s,2},\ldots,c_{s,d})$ satisfying $c_{s,j}-c_{j}\in \{b\cdot 2^{-k}:b=0,1,\ldots,4\}$. 
In other words, the center of $\boldsymbol{I}_s$ can be written as
\[
\boldsymbol{c}_{\boldsymbol{I}_s}=\boldsymbol{c}_{\boldsymbol{I}}+2^{-k}\boldsymbol{b}_s,
\]
for some $\boldsymbol{b}_s=(b_{s,1},b_{s,2},\ldots,b_{s,d})\in \{0,1,2,3,4\}^d$. 

This gives an injection from $\{\mathscr{D}_{k,s}:s=1,2,\ldots,R_k\}$ to $\{0,1,2,3,4\}^d$, which provides us the upper bound of the number of partitions $S_k\le 5^d$.
\end{proof}
We can now estimate $\E[\|\mathcal{M}_k-\mathcal{M}_{k-1}\|_{\ell^q}^p]$ by
\begin{align*}
\E[\|\mathcal{M}_k-\mathcal{M}_{k-1}\|_{\ell^q}^p]\le |S_k|^{p-1} \sum_{s=1}^{S_k}\E\bigg[\bigg\|\sum_{\boldsymbol{I}\in \mathscr{D}_{k,s}}\mathfrak{D}_{\boldsymbol{I}}^\varphi\bigg\|_{\ell^q}^p\bigg]\lesssim \sum_{s=1}^{S_k}\sum_{\boldsymbol{I}\in \mathscr{D}_{k,s}}\E[\|\mathfrak{D}_{\boldsymbol{I}}^\varphi\|_{\ell^q}^p]=\sum_{\boldsymbol{I}\in \mathscr{D}_k}\E[\|\mathfrak{D}_{\boldsymbol{I}}^\varphi\|_{\ell^q}^p],
\end{align*}
where the second inequality follows from \eqref{eq:def-ind-Mtype-scalar} and the implicit constant is independent of $k$.

\medskip
{\flushleft \bf Step 3. Applying the local estimate and conclusion.}
\medskip

\begin{proposition}\label{prop:localization}
For each $k\ge 1$ and $\boldsymbol{I}\in \mathscr{D}_k$, we have the estimate 
\[
\E[\|\mathfrak{D}_{\boldsymbol{I}}^\varphi\|_{\ell^q}^p]\lesssim k^{4dp} 2^{k(-dp+dp/q+\tau p/2)}\sup_{\bz\in \T^d}\bigg[\prod_{j=1}^k \big(\E[|X_j(\bz)|^p]\big)\bigg].
\]
\end{proposition}

\begin{lemma}\label{lemma:p-moment-Xj}
For any $0<p<\infty$, we have
\[
\lim_{j\to \infty}\sup_{\bz\in \T^d}\E[|X_j(\bz)|^p]=2^{p(p-1)\gamma^2/2}.
\]
\end{lemma}
\begin{proof}
Fix $\bz\in \T^d$. We have
\[
\E[|X_j(\bz)|^p]=\E\big[\exp\big(p\gamma \psi_j(\bz)-\frac{p\gamma^2}{2}\E[\psi_j(\bz)^2]\big)\big]=\exp\bigg(\frac{p(p-1)\gamma^2}{2}\E[\psi_j(\bz)^2]\bigg).
\]
Recalling that the kernel function of $\psi_j$ is $\mathcal{K}_j$, hence
\[
\E[|X_j(\bz)|^p]=\exp\bigg(\frac{p(p-1)\gamma^2}{2}\mathcal{K}_j(\boldsymbol{1})\bigg),
\]
which is independent of $\bz$. Combining (P2) in Proposition~\ref{prop:smooth-decomposition-log-correlated-field}, we have
\[
\lim_{j\to \infty}\E[|X_j(\bz)|^p]=\lim_{j\to \infty}\exp\bigg(\frac{p(p-1)\gamma^2}{2}\mathcal{K}_j(\boldsymbol{1})\bigg)=\exp\bigg(\frac{p(p-1)\gamma^2}{2}\log 2\bigg)=2^{p(p-1)\gamma^2/2}.
\]
This proves the desired limit.
\end{proof}

Lemma~\ref{lemma:p-moment-Xj} implies that for any $\varepsilon>0$, there exists a constant $M_{\varepsilon}>0$, such that 
\[
\sup_{\bz\in \T^d}\bigg[\prod_{j=1}^k\E[|X_j(\bz)|^p]\bigg]\le M_{\varepsilon}2^{kp(p-1)\gamma^2/2+\varepsilon},\quad \forall k\ge 1.
\]

By Proposition~\ref{prop:localization}, we have
\begin{align*}
\sum_{\boldsymbol{I}\in \mathscr{D}_k}\E[\|\mathfrak{D}_{\boldsymbol{I}}^\varphi\|_{\ell^q}^p]&\lesssim 2^{kd}\cdot k^{4dp} 2^{k(-dp+dp/q+\tau p/2)}\sup_{\bz\in \T^d}\bigg[\prod_{j=1}^k \big(\E[|X_j(\bz)|^p]\big)\bigg]\\
&\lesssim k^{4dp} 2^{k(-(p-1)d+dp/q+\tau p/2)}\cdot 2^{kp(p-1)\gamma^2/2+\varepsilon}.
\end{align*}
Combining this with \eqref{eq:decoupling-martingale}, we obtain
\begin{align}
\notag \E[\|\mathcal{M}_m\|_{\ell^q}^p]&\lesssim \sum_{k=1}^m k^{4dp} 2^{k(-(p-1)d+dp/q+\tau p/2)}\cdot 2^{kp(p-1)\gamma^2/2+\varepsilon}\\
&\lesssim \sum_{k=1}^\infty k^{4dp} 2^{k(-(p-1)d+dp/q+\tau p/2)}\cdot 2^{kp(p-1)\gamma^2/2+\varepsilon}. \label{eq:series-moment-estimate-Mm}
\end{align}
The series above converges if and only if $-(p-1)d+dp/q+\tau p/2+p(p-1)\gamma^2/2+\varepsilon<0$. Define 
\[
\zeta(p)=2d+\gamma^2-\bigg(\frac{2d}{p}+p\gamma^2\bigg).
\]
Thus, for any $\tau>0$ satisfying
\[
\tau<\sup_{1\le p\le 2}\zeta(p),
\]
we can choose $q$ and $\varepsilon$, such that the series \eqref{eq:series-moment-estimate-Mm} converges. An elementary computation shows that
\begin{itemize}
\item If $0<\gamma<\sqrt{2d}/2$, then $\sup_{1\le p\le 2}\zeta(p)=\zeta(2)=d-\gamma^2$.
\item If $\sqrt{2d}/2\le \gamma<\sqrt{2d}$, then $\sup_{1\le p\le 2}\zeta(p)=\zeta(\sqrt{2d}/\gamma)=(\sqrt{2d}-\gamma)^2$.
\end{itemize}
This proves Proposition~\ref{prop:uniform-boundedness-Fourier-Lebesgue-norm}.

We now turn to the proof of Proposition~\ref{prop:localization}, which spans the next two subsections.

\subsection{Asymptotic estimates for the $p$-moments of derivatives of $X_j$}

This subsection is devoted to the proof of Lemma~\ref{lemma:p-moment-higher-derivatives-Xj}, which establishes the uniform boundedness of the $p$-moments of higher derivatives of the random processes $X_j$ defined in \eqref{eq:def-Xj}. This estimate is essential in the proof of Proposition~\ref{prop:localization}.

\begin{lemma}\label{lemma:p-moment-higher-derivatives-Xj}
For any $0<p<\infty$, we have
\[
\sup_{j}\sup_{\bz\in \T^d}\frac{\E[|D^{\boldsymbol{\alpha}} X_j(\bz)|^p]}{j^{2|\boldsymbol{\alpha}|p}2^{j|\boldsymbol{\alpha}|p}}<\infty.
\]
\end{lemma}

The computation of the higher partial derivatives of $X_j$ relies on a special case of multivariate Faà di Bruno identity as follows (for general case, see e.g.,~\cite[Appendix~A.2]{MV11}).
\begin{lemma}\label{lemma:Faa-di-Bruno}
	Let $U$ be a domain in $\R^d$ and $f\in C^\infty(U)$. Then for any multi-index $\boldsymbol{\alpha}\in \N^d$, we have
	\[
		\partial^{\boldsymbol{\alpha}} e^f=e^f\sum_{\sum_\ell m_\ell \boldsymbol{\beta}_\ell=\boldsymbol{\alpha}} \frac{\boldsymbol{\alpha}!}{\prod_\ell m_\ell!}\prod_{\ell} \bigg(\frac{\partial^{\boldsymbol{\beta}_\ell}f}{\boldsymbol{\beta}_\ell!}\bigg)^{m_\ell}.
	\]
\end{lemma}

\begin{proof}[Proof of Lemma~\ref{lemma:p-moment-higher-derivatives-Xj}]
By definition,
\[
D^{\boldsymbol{\alpha}} X_j(\bz)=D^{\boldsymbol{\alpha}}\Big(\exp\big(\gamma \psi_j(\bz)-\frac{\gamma^2}{2}\E[\psi_j(\bz)^2]\big)\Big)=\exp\big(\gamma^2 \mathcal{K}_j(\boldsymbol{1})/2\big)\cdot D^{\boldsymbol{\alpha}}\Big(\exp\big(\gamma\psi_j(\bz)\big)\Big).
\]
By Lemma~\ref{lemma:Faa-di-Bruno}, we have
\[
D^{\boldsymbol{\alpha}}\exp\big(\gamma \psi_j(\bz)\big)=\exp(\gamma\psi_j(\bz))\sum_{\sum_\ell m_\ell\boldsymbol{\beta}_\ell=\boldsymbol{\alpha}} \frac{\boldsymbol{\alpha}!}{\prod_\ell m_\ell!} \prod_\ell \bigg(\frac{D^{\boldsymbol{\beta}_\ell}\psi_j(\bz)}{\boldsymbol{\beta}_\ell!}\bigg)^{m_\ell}.
\]
Let $N(\boldsymbol{\alpha})$ denote the total number of unordered partitions $\sum_\ell m_\ell \boldsymbol{\beta}_\ell=\boldsymbol{\alpha}$. Then
\begin{align*}
\E[|D^{\boldsymbol{\alpha}}X_j(\bz)|^p]&=\E\bigg[\bigg|\exp(\gamma\psi_j(\bz))\sum_{\sum_\ell m_\ell\boldsymbol{\beta}_\ell=\boldsymbol{\alpha}}\frac{\boldsymbol{\alpha}!}{\prod_\ell m_\ell}\prod_\ell \bigg(\frac{D^{\boldsymbol{\beta}_\ell}\psi_j(\bz)}{\boldsymbol{\beta}_\ell!}\bigg)^{m_\ell}\bigg|^p\bigg]\\
&\le N(\boldsymbol{\alpha})^{p-1}\sum_{\sum_\ell m_\ell \boldsymbol{\beta}_\ell=\boldsymbol{\alpha}}\frac{(\boldsymbol{\alpha}!)^p}{\prod_\ell (m_\ell!)^p} \E\bigg[\bigg|\exp(\gamma\psi_j(\bz))\prod_{\ell} \bigg(\frac{D^{\boldsymbol{\beta}_\ell}\psi_j(\bz)}{\boldsymbol{\beta}_\ell!}\bigg)^{m_\ell}\bigg|^p\bigg].
\end{align*}

Applying the rough estimate $\prod_\ell m_\ell!\ge 1$ and $\boldsymbol{\beta}_\ell!\ge 1$, we have
\begin{align*}
\E[|D^{\boldsymbol{\alpha}}X_j(\bz)|^p]&\le N(|\boldsymbol{\alpha}|)^p (\boldsymbol{\alpha}!)^p \sup_{\sum_\ell m_\ell \boldsymbol{\beta}_\ell=\boldsymbol{\alpha}} \E\bigg[\bigg|\exp(\gamma\psi_j(\bz))\prod_{\ell} \bigg(\frac{D^{\boldsymbol{\beta}_\ell}\psi_j(\bz)}{\boldsymbol{\beta}_\ell!}\bigg)^{m_\ell}\bigg|^p\bigg]\\
&\le N(|\boldsymbol{\alpha}|)^p (\boldsymbol{\alpha}!)^p \sup_{\sum_\ell m_\ell \boldsymbol{\beta}_\ell=\boldsymbol{\alpha}} \E\Big[\big|\exp(\gamma\psi_j(\bz))\prod_{\ell} |D^{\boldsymbol{\beta}_\ell}\psi_j(\bz)|^{m_\ell}\big|^p\Big].
\end{align*}

For each fixed unordered partition $\sum_\ell m_\ell \boldsymbol{\beta}_\ell=\boldsymbol{\alpha}$, we apply the generalized H\"{o}lder's inequality. Let $L$ be the number of distinct multi-indices $\{\boldsymbol{\beta}_\ell\}_{\ell}$ in the partition. We make the specific choice
\[
p_0 = p_\ell = (L+1)p \quad \text{for all } \ell.
\]
This choice of exponents satisfies the condition 
\[
\frac{1}{p_0} + \sum_{\ell} \frac{1}{p_\ell} = \frac{1}{p}.
\]
We then have
\[
\E\Big[\big|\exp(\gamma\psi_j(\bz))\prod_{\ell} |D^{\boldsymbol{\beta}_\ell}\psi_j(\bz)|^{m_\ell}\big|^p\Big]\le \E[\exp(p_0\gamma\psi_j(\bz))]^{p/p_0}\cdot\prod_\ell \E[|D^{\boldsymbol{\beta}_\ell}\psi_j(\bz)|^{m_\ell p_\ell}]^{p/p_\ell}.
\]
Since $L$ can be bounded by $|\boldsymbol{\alpha}|$ and $p_0=(L+1)p$, we have
\begin{align*}
\E[\exp(p_0\gamma\psi_j(\bz))]^{p/p_0}=\exp\Big(p_0p\gamma^2\mathcal{K}_j(\boldsymbol{1})/2\Big)&=\exp\Big((L+1)p^2\gamma^2\mathcal{K}_j(\boldsymbol{1})/2\Big)\\
&\le \exp\Big((|\boldsymbol{\alpha}|+1)p^2\gamma^2\mathcal{K}_j(\boldsymbol{1})/2\Big).
\end{align*}
Additionally, we know that $\lim\limits_{j\to \infty}\mathcal{K}_j(\boldsymbol{1})=\log 2$, so $\sup_j \mathcal{K}_j(\boldsymbol{1})\le C$ for some constant $C>0$. Thus
\[
\E[\exp(p_0\gamma\psi_j(\bz))]^{p/p_0}\le \exp\Big(C(|\boldsymbol{\alpha}|+1)p^2\gamma^2/2\Big).
\]

Now we estimate the product term $\prod_\ell \E[|D^{\boldsymbol{\beta}_\ell}\psi_j(\bz)|^{m_\ell p_\ell}]^{p/p_\ell}$. By (P4) in Proposition~\ref{prop:smooth-decomposition-log-correlated-field}, there exists an increasing function $\Theta$, such that
\[
\E\big[|D^{\boldsymbol{\beta}_j}\psi_j(\bz)|^{m_\ell p_\ell}\big]\le \Theta(m_\ell p_\ell) j^{2|\boldsymbol{\beta}_\ell |m_\ell p_\ell}2^{j|\boldsymbol{\alpha}|m_\ell p_\ell}.
\]
Since $\sum_\ell m_\ell |\boldsymbol{\beta}_\ell|=|\boldsymbol{\alpha}|$, we have
\[
\prod_{\ell}\Big(\E\big[\big|D^{\boldsymbol{\beta}_{\ell}}\psi_j(\bz)\big|^{m_\ell p_\ell}\big]\Big)^{p/p_\ell}\le \bigg[\prod_\ell \Theta(m_\ell p_\ell)\bigg]^{p/p_\ell} j^{2|\boldsymbol{\alpha}|p}2^{j|\boldsymbol{\alpha}|p}.
\]
Therefore,
\[
\frac{\E[|D^{\boldsymbol{\alpha}} X_j(\bz)|^p]}{j^{2|\boldsymbol{\alpha}|p}2^{j|\boldsymbol{\alpha}|p}}\le N(\boldsymbol{\alpha})^p(\boldsymbol{\alpha}!)^p\sup_{\sum_\ell m_\ell \boldsymbol{\beta}_\ell=\boldsymbol{\alpha}}\exp(p_0p\gamma^2\E[\psi_j(\bz)^2]/2)\cdot\bigg[\prod_\ell \Theta(m_\ell p_\ell)\bigg]^{p/p_\ell}.
\]
Again, recalling that $p_\ell=(L+1)p\le (|\boldsymbol{\alpha}|+1)p$ and $m_\ell\le |\boldsymbol{\alpha}|$, we have
\[
\prod_{\ell}\Big(\E\big[\big|D^{\boldsymbol{\beta}_{\ell}}\psi_j(\bz)\big|^{m_\ell p_\ell}\big]\Big)^{p/p_\ell}\le \big(\Theta(|\boldsymbol{\alpha}|\cdot (|\boldsymbol{\alpha}|+1)p)\big)^{p\cdot \sum_\ell p_\ell^{-1}}j^{2|\boldsymbol{\alpha}|p}2^{j|\boldsymbol{\alpha}|p},
\]
where we use the fact that $\Theta$ is increasing. Since $p\cdot \sum_{\ell}p_\ell^{-1}=p\cdot \frac{L}{(L+1)p}\le 1$ and $I\ge 1$, we have
\[
\prod_{\ell}\Big(\E\big[\big|D^{\boldsymbol{\beta}_{\ell}}\psi_j(\bz)\big|^{m_\ell p_\ell}\big]\Big)^{p/p_\ell}\le \Theta(|\boldsymbol{\alpha}|\cdot (|\boldsymbol{\alpha}|+1)p)j^{2|\boldsymbol{\alpha}|p}2^{j|\boldsymbol{\alpha}|p}.
\]
Therefore,
\[
\frac{\E[|D^{\boldsymbol{\alpha}}X_j(\bz)|^p]}{j^{2|\boldsymbol{\alpha}|p}2^{j|\boldsymbol{\alpha}|p}}\le N(\boldsymbol{\alpha})^p(\boldsymbol{\alpha}!)^p \exp\Big(C(|\boldsymbol{\alpha}|+1)p^2\gamma^2/2\Big)\cdot \Theta(|\boldsymbol{\alpha}|\cdot (|\boldsymbol{\alpha}|+1)p),
\]
where the right hand side of the inequality is a constant independent of $j$ and $\bz$. This finishes our proof.
\end{proof}

\subsection{Proof of Proposition~\ref{prop:localization}} To motivate the strategy of the forthcoming proof, we begin by discussing an one dimensional deterministic model example that explains the key analytic mechanism.

{\flushleft \bf Heuristic.} Consider an oscillatory integral of the form
\[
\int_I f(x) e^{2\pi i n x} \, \mathrm{d}x,
\]
where $I \subset \R$ is an interval of length $|I| = 2^{-k}$ and $n \in \Z$ denotes the frequency parameter. When $|n| \ll 2^k$, the exponential factor $e^{2\pi i n x}$ varies slowly over $I$ and can be regarded as approximately constant. In this regime, the integral is mainly determined by the average value of $f$ on $I$. In contrast, when $|n| \gg 2^k$, the rapid oscillations of $e^{2\pi i n x}$ within $I$ cause the integral to decay. In the high-frequency regime, the decay of the integral can be effectively estimated by repeated integration by parts, which is a standard technique for oscillatory integrals. 

Recall that
\[
\mathfrak{D}_{\boldsymbol{I}}^\varphi(\boldsymbol{n})=\langle \boldsymbol{n}\rangle^{\tau/2}\int_{\T^d}\varphi_{\boldsymbol{I}}(\bz)\bigg[\prod_{j=1}^{k-1}X_j(\bz)\bigg]\mathring{X}_k(\bz)\cdot \bz^{\boldsymbol{n}}\,\mathrm{d}\bz,
\]
where $\mathring{X}_k(\bz)=X_k(\bz)-1,\forall \bz\in \T^d$. For simplicity of our analysis, we can bound the centered process $\mathring{X}_k$ by the original process $X_k$. The validity of this approach, up to a multiplicative constant, is justified by the following lemma.
\begin{lemma}\label{lemma:substitue-Xk-ring}
	Let $k\ge 1$ be integer. For any $1\le p<\infty$, there exists a constant $C(p)>0$ depending only on $p$, such that 
	\[
	\E[|(D^{\boldsymbol{\alpha}}\mathring{X}_k)(\mathcal{E}(\bt+\boldsymbol{c_I}))|^p]\le C(p)\E[|(D^{\boldsymbol{\alpha}}X_k)(\mathcal{E}(\bt+\boldsymbol{c_I}))|^p].
	\]
\end{lemma}
\begin{proof}
Let $\boldsymbol{\alpha}\in \N^d$ be a multi-index.
\begin{itemize}
	\item If $\boldsymbol{\alpha}\ne \boldsymbol{0}$, then 
	\[
	(D^{\boldsymbol{\alpha}}\mathring{X}_k)(\mathcal{E}(\bt+\boldsymbol{c_I}))=(D^{\boldsymbol{\alpha}}X_k)(\mathcal{E}(\bt+\boldsymbol{c_I})).
	\]
	In this case, the equality
	\[
	\E[|(D^{\boldsymbol{\alpha}}\mathring{X}_k)(\mathcal{E}(\bt+\boldsymbol{c_I}))|^p]=\E[|(D^{\boldsymbol{\alpha}}X_k)(\mathcal{E}(\bt+\boldsymbol{c_I}))|^p]
	\]
	holds.
	\item  If $\boldsymbol{\alpha}=\boldsymbol{0}$, then 
	\[
	\E[|\mathring{X}_k(\mathcal{E}(\bt+\boldsymbol{c_I}))|^p]=\E[|X_k(\mathcal{E}(\bt+\boldsymbol{c_I}))-1|^p]\le 2^{p-1}\big(\E[|X_k(\mathcal{E}(\bt+\boldsymbol{c_I}))|^p]+1\big).
	\]
	For $1\le p<\infty$, we have
	\[
	\E[|X_k(\mathcal{E}(\bt+\boldsymbol{c_I}))|^p]=\exp\big(\frac{p(p-1)\gamma^2}{2}K_k(\boldsymbol{1})\big)\ge 1.
	\]
	Thus
	\[
	\E[|\mathring{X}_k(\mathcal{E}(\bt+\boldsymbol{c_I}))|^p]\le 2^{p-1}\cdot 2\E[|X_k(\mathcal{E}(\bt+\boldsymbol{c_I}))|^p]=2^p\E[|X_k(\mathcal{E}(\bt+\boldsymbol{c_I}))|^p].
	\]
\end{itemize}

If we define $C(p)=2^{p},\forall p\ge 1$, then for any $1\le p<\infty$,
\[
\E[|(D^{\boldsymbol{\alpha}}\mathring{X}_k)(\mathcal{E}(\bt+\boldsymbol{c_I}))|^p]\le C(p)\E[|(D^{\boldsymbol{\alpha}}X_k)(\mathcal{E}(\bt+\boldsymbol{c_I}))|^p],
\]
which proves this lemma.
\end{proof}

\subsubsection{High-low frequency decomposition}

Firstly, we divide $\E[\|\mathfrak{D}_{\boldsymbol{I}}^\varphi\|_{\ell^q}^p]$ into low frequency and high frequency parts. Since $q/p\le 1$, we have
\begin{align*}
\E[\|\mathfrak{D}_{\boldsymbol{I}}^\varphi\|_{\ell^q}^p]&=\E\bigg[\bigg\{\sum_{\boldsymbol{n}\in \Z^d} |\mathfrak{D}_{\boldsymbol{I}}^\varphi (\boldsymbol{n})|^q\bigg\}^{p/q}\bigg]\\
&\le \E\bigg[\bigg\{\sum_{\substack{\boldsymbol{n}\in \Z^d,\\|\boldsymbol{n}|\le 2^k}} |\mathfrak{D}_{\boldsymbol{I}}^\varphi (\boldsymbol{n})|^q\bigg\}^{p/q}\bigg]+\E\bigg[\bigg\{\sum_{\substack{\boldsymbol{n}\in \Z^d,\\|\boldsymbol{n}|>2^k}} |\mathfrak{D}_{\boldsymbol{I}}^\varphi (\boldsymbol{n})|^q\bigg\}^{p/q}\bigg]\\
&=I_{\textnormal{low}}+I_{\textnormal{high}}.
\end{align*}
\subsubsection{Bound $I_{\textnormal{low}}$}
Since the oscillation is negligible in this case, we can simply bound $|\mathfrak{D}_{\boldsymbol{I}}^\varphi(\boldsymbol{n})|$ by
\[
	|\mathfrak{D}_{\boldsymbol{I}}^\varphi(\boldsymbol{n})|\le \langle \boldsymbol{n}\rangle^{\tau/2}\int_{[-1/2,1/2)^d} \bigg|\varphi_{\boldsymbol{I}}(\mathcal{E}(\bt))\bigg[\prod_{j=1}^{k-1}X_j(\mathcal{E}(\bt))\bigg]\mathring{X}_k(\mathcal{E}(\bt))\bigg|\,\mathrm{d}\bt.
\]
Note that the integral term is independent of $\boldsymbol{n}$, hence we can take it out of the summation over $\boldsymbol{n}$. Therefore,
\begin{align*}
\sum_{\substack{\boldsymbol{n}\in \Z^d,\\|\boldsymbol{n}|\le 2^k}} |\mathfrak{D}_{\boldsymbol{I}}^\varphi (\boldsymbol{n})|^q&\le \bigg\{\int_{[-1/2,1/2)^d} \bigg|\varphi_{\boldsymbol{I}}(\mathcal{E}(\bt))\bigg[\prod_{j=1}^{k-1}X_j(\mathcal{E}(\bt))\bigg]\mathring{X}_k(\mathcal{E}(\bt))\bigg|\,\mathrm{d}\bt\bigg\}^q\cdot \bigg\{\sum_{\substack{\boldsymbol{n}\in \Z,\\|\boldsymbol{n}|\le 2^k}} \langle \boldsymbol{n}\rangle^{\tau q/2}\bigg\}\\
&\lesssim 2^{k\tau q/2}\cdot 2^{kd}\cdot \bigg\{\int_{[-1/2,1/2)^d} \bigg|\varphi_{\boldsymbol{I}}(\mathcal{E}(\bt))\bigg[\prod_{j=1}^{k-1}X_j(\mathcal{E}(\bt))\bigg]\mathring{X}_k(\mathcal{E}(\bt))\bigg|\,\mathrm{d}\bt\bigg\}^q. 
\end{align*}
It follows that
\begin{align*}
I_{\textnormal{low}}&\lesssim 2^{k\tau p/2}\cdot 2^{kdp/q}\cdot \E\bigg[\bigg\{\int_{[-1/2,1/2)^d} \bigg|\varphi_{\boldsymbol{I}}(\mathcal{E}(\bt))\bigg[\prod_{j=1}^{k-1}X_j(\mathcal{E}(\bt))\bigg]\mathring{X}_k(\mathcal{E}(\bt))\bigg|\,\mathrm{d}\bt\bigg\}^p\bigg]\\
&\le 2^{k(\tau p/2+dp/q)}\cdot \bigg\{\int_{[-1/2,1/2)^d} |\varphi_{\boldsymbol{I}}(\mathcal{E}(\bt))|\bigg\{\E\bigg[\bigg|\bigg[\prod_{j=1}^{k-1}X_j(\mathcal{E}(\bt))\bigg]\mathring{X}_k(\mathcal{E}(\bt))\bigg|^p\bigg]\bigg\}^{1/p}\,\mathrm{d}\bt\bigg\}^p,
\end{align*}
where the second inequality holds by Minkowski's inequality. Since $\varphi_{\boldsymbol{I}}$ is supported on $2^{-k}([-1,1)^d+\boldsymbol{c_I})$, which has volume $2^{-(k-1)d}$, we have
\[
I_{\textnormal{low}}\lesssim 2^{k(-dp+\tau p/2+dp/q)}\cdot \sup_{\bz\in \T^d}\E\bigg[\prod_{j=1}^{k-1} |X_j(\bz)|^p |\mathring{X}_k(\bz)|^p\bigg].
\]
Using the independence of $X_j$'s and Lemma~\ref{lemma:substitue-Xk-ring}, we obtain
\begin{equation}\label{eq:bound-I-low}
I_{\textnormal{low}}\lesssim 2^{k(-dp+\tau p/2+dp/q)}\cdot \sup_{\bz\in \T^d}\bigg[\prod_{j=1}^{k}\E[|X_j(\bz)|^p]\bigg].
\end{equation}

\subsubsection{Bound $I_{\textnormal{high}}$}\label{sssec:Bound-Ihigh}
We will divide the estimate of $I_{\textnormal{high}}$ into two steps for clarity.
\medskip
{\flushleft \bf Step 1. Integration by parts.}
\medskip

Our main goal in bounding $I_{\textnormal{high}}$ is to exploit the rapid oscillation of the term $e_{\boldsymbol{n}}(\bt)$ for large $|\boldsymbol{n}|$. As mentioned earlier, the standard analytical tool for estimating such oscillatory integrals is integration by parts. In our multi-dimensional setting, this is achieved via Green's second identity.
\begin{lemma}[Green's second identity]\label{lemma:Green's-second-id}
Let $\Omega$ be a region in $\R^d$, and let $f,g\in C^\infty(\Omega)$. We have
\begin{equation}\label{eq:Green's-second-id}
\int_{\Omega}(\Delta f)\cdot g\,\mathrm{d}\bt=\int_{\Omega} f\cdot (\Delta g)\,\mathrm{d}\bt+\int_{\partial \Omega} \bigg(\frac{\partial f}{\partial \boldsymbol{n}}\cdot g-f\cdot \frac{\partial g}{\partial \boldsymbol{n}}\bigg)\,\mathrm{d}\sigma,
\end{equation}
where $\frac{\partial f}{\partial \boldsymbol{n}}$ is the normal derivative of $f$ on the boundary $\partial \Omega$ and $\mathrm{d}\sigma$ is the surface measure on $\partial \Omega$.
\end{lemma}
Note that
\[
\Delta e_{\boldsymbol{n}}=-4\pi^2|\boldsymbol{n}|^2e_{\boldsymbol{n}}.
\]
If $f$ is a $C^\infty(\R^d)$ function supported on $\Omega$, then the boundary term in \eqref{eq:Green's-second-id} vanishes. In this case, we have
\begin{equation}\label{eq:Green's-second-id-2}
\int_{\Omega}(\Delta f)(\bt)\cdot e_{\boldsymbol{n}}(\bt)\,\mathrm{d}\bt=\int_\Omega f(\bt)\Delta e_{\boldsymbol{n}}(\bt)\,\mathrm{d}\bt=-4\pi^2|\boldsymbol{n}|^2\int_{\Omega} f(\bt)e_{\boldsymbol{n}}(\bt)\,\mathrm{d}\bt.
\end{equation}

For each $k\ge 1$ and $\boldsymbol{I}\in \mathscr{D}_k$, define the function $\mathcal{X}_{\boldsymbol{I}}^\varphi\colon \R^d\to \R$ by
\[
\mathcal{X}_{\boldsymbol{I}}^\varphi(\bt)\coloneq \varphi_k(\mathcal{E}(\bt))\bigg[\prod_{j=1}^{k-1}X_j(\mathcal{E}(\bt+\boldsymbol{c_I}))\bigg]\mathring{X}_k(\mathcal{E}(\bt+\boldsymbol{c_I})),\quad \forall \bt\in [-1/2,1/2]^d, 
\]
and extend it to $\R^d$ by zero. Then $\mathcal{X}_{\boldsymbol{I}}^\varphi$ is a $C^\infty(\R^d)$ function supported on $2^{-k}[-1,1)^d$. 

\begin{lemma}\label{lemma:integration-by-parts}
Let $\tau>0$. For each $k\ge 1$, $\boldsymbol{I}\in \mathscr{D}_k$ and $\boldsymbol{n}\in \Z^d\setminus\{\boldsymbol{0}\}$, we have
\[
\mathfrak{D}_{\boldsymbol{I}}^\varphi(\boldsymbol{n})=\langle \boldsymbol{n}\rangle^{\tau/2}e_{\boldsymbol{n}}(\boldsymbol{c}_{\boldsymbol{I}})(-4\pi^2|\boldsymbol{n}|^2)^{-d}\sum_{|\boldsymbol{\alpha}|=d}\int_{2^{-k}[-1,1)^d}(\partial^{2\boldsymbol{\alpha}}\mathcal{X}_{\boldsymbol{I}}^\varphi)(\bt)e_{\boldsymbol{n}}(\bt)\,\mathrm{d}\bt.
\]
\end{lemma}
\begin{proof}
Note that
\begin{align*}
\mathfrak{D}_{\boldsymbol{I}}^\varphi(\boldsymbol{n})&=\langle \boldsymbol{n}\rangle^{\tau/2}\int_{\T^d}\varphi_{\boldsymbol{I}}(\bz) \bigg[\prod_{j=1}^{k-1}X_j(\bz)\bigg]\mathring{X}_k(\bz)\cdot \bz^{\boldsymbol{n}}\,\boldsymbol{m}(\mathrm{d}\bz)\\
&=\langle \boldsymbol{n}\rangle^{\tau/2}\int_{[-1/2,1/2)^d}\varphi_{\boldsymbol{I}}(\mathcal{E}(\bt))\bigg[\prod_{j=1}^{k-1}X_j(\mathcal{E}(\bt))\bigg] \mathring{X}_k(\mathcal{E}(\bt))\cdot e_{\boldsymbol{n}}(\bt)\,\mathrm{d}\bt\\
&=\langle \boldsymbol{n}\rangle^{\tau/2}\int_{[-1/2,1/2)^d+\boldsymbol{c_I}}\varphi_{k}(\mathcal{E}(\bt-\boldsymbol{c_I}))\bigg[\prod_{j=1}^{k-1}X_j(\mathcal{E}(\bt))\bigg] \mathring{X}_k(\mathcal{E}(\bt))\cdot e_{\boldsymbol{n}}(\bt)\,\mathrm{d}\bt.
\end{align*}
By changing of variables, we have
\begin{align*}
\mathfrak{D}_{\boldsymbol{I}}^\varphi(\boldsymbol{n})&=\langle \boldsymbol{n}\rangle^{\tau/2}e_{\boldsymbol{n}}(\boldsymbol{c}_{\boldsymbol{I}})\int_{[-1/2,1/2)^d}\varphi_{k}(\mathcal{E}(\bt)) \bigg[\prod_{j=1}^{k-1}X_j(\mathcal{E}(\bt+\boldsymbol{c_I}))\bigg] \mathring{X}_k(\mathcal{E}(\bt+\boldsymbol{c_I}))\cdot e_{\boldsymbol{n}}(\bt)\,\mathrm{d}\bt\\
&=\langle \boldsymbol{n}\rangle^{\tau/2}e_{\boldsymbol{n}}(\boldsymbol{c}_{\boldsymbol{I}})\int_{2^{-k}[-1,1)^d}\mathcal{X}_{\boldsymbol{I}}^\varphi(\bt)e_{\boldsymbol{n}}(\bt)\,\mathrm{d}\bt.
\end{align*}
Applying \eqref{eq:Green's-second-id-2}, one can obtain
\[
\int_{2^{-k}[-1,1)^d}\mathcal{X}_{\boldsymbol{I}}^\varphi(\bt)e_{\boldsymbol{n}}(\bt)\,\mathrm{d}\bt=(-4\pi^2 |\boldsymbol{n}|^2)^{-d}\int_{2^{-k}[-1,1)^d}(\Delta^d\mathcal{X}_{\boldsymbol{I}}^\varphi)(\bt) e_{\boldsymbol{n}}(\bt)\,\mathrm{d}\bt.
\]
Therefore
\begin{align*}
\mathfrak{D}_{\boldsymbol{I}}^\varphi(\boldsymbol{n})&=\langle \boldsymbol{n}\rangle^{\tau/2}e_{\boldsymbol{n}}(\boldsymbol{c}_{\boldsymbol{I}})(-4\pi^2|\boldsymbol{n}|^2)^{-d}\int_{2^{-k}[-1,1)^d}(\Delta^d\mathcal{X}_{\boldsymbol{I}}^\varphi)(\bt)e_{\boldsymbol{n}}(\bt)\,\mathrm{d}\bt\\
&=\langle \boldsymbol{n}\rangle^{\tau/2}e_{\boldsymbol{n}}(\boldsymbol{c}_{\boldsymbol{I}})(-4\pi^2|\boldsymbol{n}|^2)^{-d}\sum_{|\boldsymbol{\alpha}|=d}\int_{2^{-k}[-1,1)^d}(\partial^{2\boldsymbol{\alpha}}\mathcal{X}_{\boldsymbol{I}}^\varphi)(\bt)e_{\boldsymbol{n}}(\bt)\,\mathrm{d}\bt.
\end{align*}
This completes the proof of this lemma.
\end{proof}

\medskip
{\flushleft \bf Step 2. Bound the integral term.}
\medskip

By Lemma~\ref{lemma:integration-by-parts}, for $\boldsymbol{n}\in \Z^d\setminus\{\boldsymbol{0}\}$, we have
\begin{align*}
|\mathfrak{D}_{\boldsymbol{I}}^\varphi(\boldsymbol{n})|&\le \langle \boldsymbol{n}\rangle^{\tau/2}(4\pi^2|\boldsymbol{n}|^2)^{-d}\sum_{|\boldsymbol{\alpha}|=d}\int_{2^{-k}[-1,1)^d}\big|(\partial^{2\boldsymbol{\alpha}}\mathcal{X}_{\boldsymbol{I}}^\varphi)(\bt)\big|\,\mathrm{d}\bt\\
&\lesssim |\boldsymbol{n}|^{\tau/2-2d}\cdot \sum_{|\boldsymbol{\alpha}|=d}\int_{2^{-k}[-1,1)^d}\big|(\partial^{2\boldsymbol{\alpha}}\mathcal{X}_{\boldsymbol{I}}^\varphi)(\bt)\big|\,\mathrm{d}\bt.
\end{align*}
Again, observe that the integral term does not depend on $\boldsymbol{n}$, so it can be factored out of the summation over $\boldsymbol{n}$, yielding
\[
\sum_{\substack{\boldsymbol{n}\in \Z^d,\\|\boldsymbol{n}|>2^k}} |\mathfrak{D}_{\boldsymbol{I}}^\varphi (\boldsymbol{n})|^q\lesssim \bigg\{\sum_{\substack{\boldsymbol{n}\in \Z^d,\\|\boldsymbol{n}|>2^k}}|\boldsymbol{n}|^{(\tau/2-2d)q}\bigg\}\cdot \bigg(\sum_{|\boldsymbol{\alpha}|=d}\int_{2^{-k}[-1,1)^d}\big|(\partial^{2\boldsymbol{\alpha}}\mathcal{X}_{\boldsymbol{I}}^\varphi)(\bt)\big|\,\mathrm{d}\bt\bigg)^q.
\]
The series about $\boldsymbol{n}$ converges for $\tau<4d$ and sufficiently large $q$, which can be bounded by $2^{k(\tau q/2-2dq+d)}$. Therefore, we have
\begin{align*}
\bigg\{\sum_{\substack{\boldsymbol{n}\in \Z^d,\\|\boldsymbol{n}|>2^k}} |\mathfrak{D}_{\boldsymbol{I}}^\varphi(\boldsymbol{n})|^q\bigg\}^{p/q}\lesssim 2^{k(\tau p/2-2dp+dp/q)}\cdot \bigg(\sum_{|\boldsymbol{\alpha}|=d}\int_{2^{-k}[-1,1)^d}\big|(\partial^{2\boldsymbol{\alpha}}\mathcal{X}_{\boldsymbol{I}}^\varphi)(\bt)\big|\,\mathrm{d}\bt\bigg)^p.
\end{align*}
Thus, applying Minkowski's inequality, we obtain
\begin{align*}
I_{\textnormal{high}}=\E\bigg[\bigg\{\sum_{\substack{\boldsymbol{n}\in \Z^d,\\|\boldsymbol{n}|>2^k}} |\mathfrak{D}_{\boldsymbol{I}}^\varphi(\boldsymbol{n})|^q\bigg\}^{p/q}\bigg]&\lesssim  2^{k(\tau p/2-2dp+dp/q)}\bigg\{\int_{2^{-k}[-1,1)^d} \sum_{|\boldsymbol{\alpha}|=d}\big(\E[|(\partial^{2\boldsymbol{\alpha}}\mathcal{X}_{\boldsymbol{I}}^\varphi)(\bt)|^p]\big)^{1/p}\,\mathrm{d}\bt\bigg\}^p\\
&\lesssim 2^{k(\tau p/2-2dp+dp/q)}\cdot 2^{-kdp}\cdot  \sup_{|\boldsymbol{\alpha}|=2d}\sup_{\bt\in 2^{-k}[-1,1)^d}\E[|\partial^{\boldsymbol{\alpha}}\mathcal{X}_{\boldsymbol{I}}^\varphi(\bt)|^p].
\end{align*}

The following lemma provides a bound for the $p$-norm of the higher derivatives of $\mathcal{X}_{\boldsymbol{I}}^\varphi$.
\begin{lemma}\label{lemma:p-norm-h-derivatives-of-XIphi}
For any $1\le p<\infty$, and any integer $A\ge 0$, we have
\[
\sup_{|\boldsymbol{\alpha}|=A}\sup_{\bt\in 2^{-k}[-1,1)^d}\big(\E[|\partial^{\boldsymbol{\alpha}}\mathcal{X}_{\boldsymbol{I}}^\varphi(\bt)|^p]\big)^{1/p}\lesssim \sup_{\bz\in \T^d}\bigg[\prod_{1\le j\le k}\big(\E[|D^{\boldsymbol{\alpha}}X_j(\bz)|^p]\big)^{1/p}\bigg]\cdot (k^2 2^k)^{A},
\]
where the implicit constant is independent of $k$ and $\boldsymbol{I}$.
\end{lemma}
Taking $A=2d$ in this lemma, we can conclude that
\begin{equation}\label{eq:bound-I-high}
I_{\textnormal{high}}\lesssim k^{4dp}\cdot 2^{k(-dp+\tau p/2+dp/q)}\cdot \sup_{\bz\in \T^d}\bigg[\prod_{1\le j\le k}\E[|D^{\boldsymbol{\alpha}}X_j(\bz)|^p]\bigg].
\end{equation}
Now we turn back to prove Lemma~\ref{lemma:p-norm-h-derivatives-of-XIphi}.
\begin{proof}[Proof of Lemma~\ref{lemma:p-norm-h-derivatives-of-XIphi}]
Since $\partial_\ell$ commutes with translation, we know that
\[
(\partial^{\boldsymbol{\alpha}}X_j(\mathcal{E}(\cdot+\boldsymbol{c_I})))(\bt)=(D^{\boldsymbol{\alpha}}X_j)(\mathcal{E}(\bt+\boldsymbol{c_I})).
\]
The same formula holds for $\mathring{X}_k$. Therefore, for a multi-index $\boldsymbol{\alpha}$ with $|\boldsymbol{\alpha}|=R$, by Leibniz rule, we have
\[
\partial^{\boldsymbol{\alpha}}\mathcal{X}_{\boldsymbol{I}}^\varphi(\bt)=\sum_{\sum_{j=0}^k \boldsymbol{\alpha}_j=\boldsymbol{\alpha}} \frac{\boldsymbol{\alpha}!}{\boldsymbol{\alpha}_0!\boldsymbol{\alpha}_1!\cdots \boldsymbol{\alpha}_k!}(D^{\boldsymbol{\alpha}_0} \varphi_k)(\mathcal{E}(\bt))\cdot \bigg[\prod_{j=1}^{k-1} (D^{\boldsymbol{\alpha}_j}X_j)(\mathcal{E}(\bt+\boldsymbol{c_I}))\bigg]\cdot (D^{\boldsymbol{\alpha}_k}\mathring{X}_k)(\mathcal{E}(\bt+\boldsymbol{c_I})),
\]
where we write $\boldsymbol{\alpha}_j=(\alpha_{j,1},\alpha_{j,2},\ldots,\alpha_{j,d})$ and $\boldsymbol{\alpha}=(\alpha_1,\alpha_2,\ldots,\alpha_d)$. For clarity, we emphasize that bold symbols such as $\boldsymbol{\alpha}_j$ denote multi-indices, while non-bold symbols like $\alpha_j$ represent scalar components.

To bound this sum, our strategy is to estimate each term by applying independence to separate the products, bounding the moments of each derivative factor using Lemma~\ref{lemma:p-moment-higher-derivatives-Xj} and Lemma~\ref{lemma:pou-scaling-Td}, and finally resumming the bounds via the multinomial theorem.

Let $\sum_{j=0}^k \boldsymbol{\alpha}_j=\boldsymbol{\alpha}$ be a fixed partition. By Lemma~\ref{lemma:substitue-Xk-ring} and independence of $X_j$'s, we have
\begin{align*}
&\E\bigg[\bigg| \bigg[\prod_{j=1}^{k-1} (D^{\boldsymbol{\alpha}_j}X_j)(\mathcal{E}(\bt+\boldsymbol{c_I}))\bigg]\cdot (D^{\boldsymbol{\alpha}_k}\mathring{X}_k)(\mathcal{E}(\bt+\boldsymbol{c_I}))\bigg|^p\bigg]\\
&=\bigg[\prod_{j=1}^{k-1}\E[|(D^{\boldsymbol{\alpha}_j}X_j)(\mathcal{E}(\bt+\boldsymbol{c_I}))|^p]\bigg]\cdot \E[|(D^{\boldsymbol{\alpha}_k}\mathring{X}_k)(\mathcal{E}(\bt+\boldsymbol{c_I}))|^p]
\\
& \lesssim \prod_{j=1}^k \E[|(D^{\boldsymbol{\alpha}_j}X_j)(\mathcal{E}(\bt+\boldsymbol{c_I}))|^p].
\end{align*}
Since $\E[|X_j(\mathcal{E}(\bt+\boldsymbol{c_I}))|^p]\ge 1$, we can insert at most $R$ terms $X_j(\mathcal{E}(\bt+\boldsymbol{c_I}))$ with $\boldsymbol{\alpha}_j=\boldsymbol{0}$, which gives
\begin{align*}
\prod_{j=1}^k \E[|(D^{\boldsymbol{\alpha}_j}X_j)(\mathcal{E}(\bt+\boldsymbol{c_I}))|^p]&=\bigg[\prod_{\substack{1\le j\le k,\\ \boldsymbol{\alpha}_j\ne \boldsymbol{0}}}\E[|(D^{\boldsymbol{\alpha}_j}X_j)(\mathcal{E}(\bt+\boldsymbol{c_I}))|^p]\bigg]\cdot \bigg[\prod_{\substack{1\le j\le k,\\ \boldsymbol{\alpha}_j=\boldsymbol{0}}}\E[|X_j(\mathcal{E}(\bt+\boldsymbol{c_I}))|^p]\bigg]\\
&\le \bigg[\prod_{\substack{1\le j\le k,\\ \boldsymbol{\alpha}_j\ne 0}} \E[|(D^{\boldsymbol{\alpha}_j}X_j)(\mathcal{E}(\bt+\boldsymbol{c_I}))|^p]\bigg]\cdot\bigg[\prod_{j=1}^k \E[|X_j(\mathcal{E}(\bt+\boldsymbol{c_I}))|^p]\bigg].
\end{align*}
By Lemma~\ref{lemma:p-moment-higher-derivatives-Xj}, for each $1\le j\le k$, there exists a constant $M(p,\boldsymbol{\alpha}_j)>0$ such that
\[
\E[|(D^{\boldsymbol{\alpha}_j}X_j)(\mathcal{E}(\bt+\boldsymbol{c_I}))|^p]\le M(p,\boldsymbol{\alpha}_j) j^{2|\boldsymbol{\alpha}_j|p}2^{j|\boldsymbol{\alpha}_j|p}.
\]
By Lemma~\ref{lemma:pou-scaling-Td}, we have
\[
\|D^{\boldsymbol{\alpha}_0}\varphi_k\|_{L^\infty(\T^d)}\lesssim 2^{k|\boldsymbol{\alpha}_0|}.
\]
Thereofre,
\begin{align*}
&\E\bigg[\bigg|(D^{\boldsymbol{\alpha}_0} \varphi_k)(\mathcal{E}(\bt))\cdot \bigg[\prod_{j=1}^{k-1} (D^{\boldsymbol{\alpha}_j}X_j)(\mathcal{E}(\bt+\boldsymbol{c_I}))\bigg]\cdot (D^{\boldsymbol{\alpha}_k}\mathring{X}_k)(\mathcal{E}(\bt+\boldsymbol{c_I}))\bigg|^p\bigg]\\
&\lesssim 2^{k|\boldsymbol{\alpha}_0|p}\bigg[\prod_{j=1}^k \E[|X_j(\mathcal{E}(\bt+\boldsymbol{c_I}))|^p]\bigg]\cdot\bigg[\prod_{\substack{1\le j\le k,\\\boldsymbol{\alpha}_j\ne 0}}(M(p,\boldsymbol{\alpha}_j)j^{2|\boldsymbol{\alpha}_j|p}2^{j|\boldsymbol{\alpha}_j|p})\bigg].
\end{align*}
Since the number of non-zero $\alpha_j$ in $\sum_j \boldsymbol{\alpha}_j=\boldsymbol{\alpha}$ is not larger than $A$, we have
\begin{align*}
&\E\bigg[\bigg|(D^{\boldsymbol{\alpha}_0} \varphi_k)(\mathcal{E}(\bt))\cdot \bigg[\prod_{j=1}^{k-1} (D^{\boldsymbol{\alpha}_j}X_j)(\mathcal{E}(\bt+\boldsymbol{c_I}))\bigg]\cdot (D^{\boldsymbol{\alpha}_k}\mathring{X}_k)(\mathcal{E}(\bt+\boldsymbol{c_I}))\bigg|^p\bigg]\\
&\lesssim 2^{k|\boldsymbol{\alpha}_0|p}\bigg[\prod_{j=1}^k \E[|X_j(\mathcal{E}(\bt+\boldsymbol{c_I}))|^p]\bigg]\cdot\bigg[\prod_{\substack{1\le j\le k,\\\boldsymbol{\alpha}_j\ne 0}}(M(p,\boldsymbol{\alpha}_j)j^{2|\boldsymbol{\alpha}_j|p}2^{j|\boldsymbol{\alpha}_j|p})\bigg].
\end{align*}
Since the number of non-zero $\alpha_j$ in $\sum_j \boldsymbol{\alpha}_j=\boldsymbol{\alpha}$ is not larger than $A$, we have
\begin{align*}
&\E\bigg[\bigg|(D^{\boldsymbol{\alpha}_0} \varphi_k)(\mathcal{E}(\bt))\cdot \bigg[\prod_{j=1}^{k-1} (D^{\boldsymbol{\alpha}_j}X_j)(\mathcal{E}(\bt+\boldsymbol{c_I}))\bigg]\cdot (D^{\boldsymbol{\alpha}_k}\mathring{X}_k)(\mathcal{E}(\bt+\boldsymbol{c_I}))\bigg|^p\bigg]\\
&\lesssim \max\Big\{\sup_{|\boldsymbol{\beta}|\le A}M(p,\boldsymbol{\beta})^A,1\Big\} \cdot \bigg[\prod_{j=1}^k \E[|X_j(\mathcal{E}(\bt+\boldsymbol{c_I}))|^p]\bigg]\cdot 2^{k|\boldsymbol{\alpha}_0|p}\bigg[ \prod_{1\le j\le k} j^{2|\boldsymbol{\alpha}_j|p}2^{j|\boldsymbol{\alpha}_j|p}\bigg]\\
&\lesssim \sup_{\bz\in \T^d}\bigg[\prod_{j=1}^k \E[|X_j(\bz)|^p]\bigg]\cdot 2^{k|\boldsymbol{\alpha}_0|p}\bigg[\prod_{1\le j\le k} j^{2|\boldsymbol{\alpha}_j|p}2^{j|\boldsymbol{\alpha}_j|p}\bigg].
\end{align*}
By Minkowski inequality, we have
\begin{align*}
&\quad \,\big(\E[|\partial^{\boldsymbol{\alpha}}\mathcal{X}_{\boldsymbol{I}}^\varphi(\bt)|^p]\big)^{1/p}\\
&\lesssim  \sup_{\bz\in \T^d}\bigg[\prod_{j=1}^k (\E[|X_j(\bz)|^p])^{1/p}\bigg]\cdot \bigg\{\sum_{\sum_{j=0}^k \boldsymbol{\alpha}_j=\boldsymbol{\alpha}}\frac{\boldsymbol{\alpha}!}{\boldsymbol{\alpha}_0!\boldsymbol{\alpha}_1!\cdots \boldsymbol{\alpha}_k!} \bigg(2^{k|\boldsymbol{\alpha}_0|}\prod_{1\le j\le k} (j^2 2^j)^{|\boldsymbol{\alpha}_j|}\bigg)\bigg\}.
\end{align*}
Note that by the definition of the factorial of a multi-index and the separability of the products, the sum can be decomposed over each dimension $\ell=1,\ldots,d$:
\begin{align*}
\sum_{\sum_{j=0}^k \boldsymbol{\alpha}_j=\boldsymbol{\alpha}}\frac{\boldsymbol{\alpha}!}{\boldsymbol{\alpha}_0!\cdots \boldsymbol{\alpha}_k!}\bigg(2^{k|\boldsymbol{\alpha}_0|}\prod_{1\le j\le k} (j^2 2^j)^{|\boldsymbol{\alpha}_j|}\bigg)=\prod_{\ell=1}^d \Bigg\{ \sum_{\sum_{j=0}^k \alpha_{j,\ell}=\alpha_\ell}\frac{\alpha_\ell!}{\alpha_{0,\ell}!\cdots \alpha_{k,\ell}!} 2^{k\alpha_{0,\ell}} \prod_{1\le j\le k} (j^2 2^j)^{\alpha_{j,\ell}} \Bigg\}.
\end{align*}
The inner sum for each $\ell$ is now a direct application of the multinomial theorem:
\begin{align*}
    \sum_{\sum_{j=0}^k \alpha_{j,\ell}=\alpha_\ell}\frac{\alpha_\ell!}{\alpha_{0,\ell}!\cdots \alpha_{k,\ell}!} 2^{k\alpha_{0,\ell}} \prod_{1\le j\le k} (j^2 2^j)^{\alpha_{j,\ell}} = \left(2^k+\sum_{j=1}^k j^2 2^j\right)^{\alpha_\ell}.
\end{align*}
Substituting this back and combining the exponents, we conclude that
\begin{equation*}
    \left(2^k+\sum_{j=1}^k j^2 2^j\right)^{\sum_{\ell=1}^d \alpha_\ell} = \left(2^k+\sum_{j=1}^k j^2 2^j\right)^{A}.
\end{equation*}

Since
\[
2^k+\sum_{j=1}^k j^2 2^j\le 2^k+k^2\sum_{j=0}^k 2^j=2^k+k^2(2^{k+1}-1)\le 3k^2 2^k,
\]
we conclude that
\[
	\big(\E[|\partial^{\boldsymbol{\alpha}}\mathcal{X}_{\boldsymbol{I}}^\varphi(\bt)|^p]\big)^{1/p}\lesssim  \sup_{\bz\in T^d}\bigg[\prod_{j=1}^k (\E[|X_j(\bz)|^p])^{1/p}\bigg]\cdot (k^22^k)^A.
\]
Taking supremum over $\bt\in 2^{-k}[-1,1)^d$, we finish the proof.
\end{proof}

\subsubsection{Conclusion}

We combine the bounds \eqref{eq:bound-I-low} and \eqref{eq:bound-I-high} to obtain
\[
\E[\|\mathfrak{D}_{\boldsymbol{I}}^\varphi\|_{\ell^q}^p]\lesssim k^{4dp}\cdot 2^{k(-dp+\tau p/2+dp/q)}\cdot \sup_{\bz\in \T^d}\bigg[\prod_{j=1}^k \E[|X_j(\bz)|^p]\bigg].
\]
This proves Proposition~\ref{prop:localization}.

\appendix

\section{Smooth partition of unity on torus}\label{appendix:pou}

In this appendix, we aim to prove Lemma~\ref{lemma:pou-scaling-Td}. The construction of partition of unity with scaling property on $\R^d$ and $\T^d$ is well-known in harmonic analysis, especially in the context of wave packet analysis and time-frequency analysis (e.g.,~\cite[Chapter~2]{Dem20}). For convenience, we shall exhibit the construction here.

Although a partition of unity on $\R^d$ and $\T^d$ could be constructed directly, we opt for a more intuitive approach by first building a one-dimensional partition of unity on $\R$ and $\T$ and then extending it to higher dimensions via a tensor product.
\medskip
{\flushleft \bf Step 1. Construction of partition of unity on $\R$.}
\medskip

A standard choice of bump function $\vartheta\in C^\infty(\R)$ with $\supp \vartheta \subset [-1,1]$ is given by
\[
\vartheta(t)=\begin{cases}
	e^{-1/(1-t^2)},& |t|\le 1,\\
	0,& |t|> 1.
\end{cases}
\]
Then $W(t)\coloneq \sum_{h\in \Z}\vartheta(t-h)$ is a periodic function with $W(t+1)=W(t)$. Note that the summation in the definition of $W$ is locally finite, and the sum is strictly positive for $t\in \R$. We can define
\[
\theta(t)\coloneq \frac{\vartheta(t)}{W(t)},\quad \forall t\in \R.
\]
Then $\theta\in C^\infty(\R)$, $\supp \theta\subset [-1,1]$, and
\[
\sum_{h\in \Z}\theta(t-h)=\sum_{h\in \Z}\frac{\vartheta(t-h)}{W(t-h)}=\frac{1}{W(t)}\sum_{h\in \Z}\vartheta(t-h)=1,\quad \forall t\in \R.
\]
Therefore, $\{\theta(\cdot-h)\}_{h\in \Z}$ forms a partition of unity on $\R$ and the support of each $\theta(\cdot-h)$ is contained in $[-1,1]+h$.

We can now introduce the scaling of $\theta$. Let $k \ge 1$. For each $I\in \widetilde{\mathscr{D}}_k^1$, we define 
\[
\theta_{I}(t)\coloneq \theta(2^k(t-c_{I})),\quad \forall t\in \R.
\]
Here, $c_{I}$ is the center of the cube $I=[0,2^{-k})+2^{-k}h$, that is, $c_{I}=2^{-k}(2^{-1}+h)$.
We claim that 
\begin{equation}\label{eq:pou-R-scaling-1D}
\sum_{I\in \widetilde{\mathscr{D}}_k^1} \theta_{I}(t)=1,\quad \forall t\in \R.
\end{equation}
Moreover, the support of $\theta_{I}$ is contained in $2I$. In fact, for any $t\in \R$, we have
\begin{align*}
\sum_{I\in \widetilde{\mathscr{D}}_k^1} \theta_{I}(t)&=\sum_{I\in \widetilde{\mathscr{D}}_k^1} \theta(2^k(t-c_{I}))=\sum_{h\in \Z}\theta(2^k(t-2^{-k}(2^{-1}+h)))=\sum_{h\in \Z}\theta((2^k t-2^{-1})-h)=1.
\end{align*}
This proves \eqref{eq:pou-R-scaling-1D}.

\medskip
{\flushleft \bf Step 2. Construction of partition of unity on $\T$.}
\medskip

The second step is to transplant the partition of unity on $\R$ to $\T$. Define 
\[
\rho(e(t))\coloneq \theta(t),\quad \forall t\in [-1/2,1/2),
\]
and
\[
\rho_k(e(t))\coloneq \theta(2^k t),\quad \forall t\in [-1/2,1/2).
\]
From the definition, we easily see that $\rho_k\in C^\infty(\T)$ and
\[
\supp \rho_k\subset \overline{B}(1,2^{-k})=e([-2^{-k},2^{-k}]).
\]

The partition of unity on $\T$ at scale $2^{-k}$ is again generated by the translation of $\rho_k$, that is,
\[
\sum_{I\in \mathscr{D}_k^1} \rho_k(e(t-c_I))=\sum_{I\in \mathscr{D}_k^1} \theta_{I}(t)=1,\quad \forall t\in [-1/2,1/2).
\]
Define $\rho_{I}(e(t))\coloneq \rho_k(e(t-c_{I}))$ for each $I\in \mathscr{D}_k^1$ and $t\in [-1/2,1/2)$. Then the equation above can be rewritten as
\[
\sum_{I\in \mathscr{D}_k^1} \rho_{I}(z)=1,\quad \forall z\in \T.
\]

Note that $\rho_k(e(\cdot-c_I))$ is nothing but the periodization of $\theta(2^k(\cdot -c_I))$, which means
\[
\rho_k(e(t-c_I))=\sum_{h\in \Z} \theta(2^k(t-c_I-h)).
\]
Taking sum over all $I\in \mathscr{D}_k^1$, we have
\[
\sum_{I\in \mathscr{D}_k^1} \rho_k(e(t-c_I))=\sum_{I\in \mathscr{D}_k^1}\sum_{h\in \Z} \theta(2^k(t-c_I-h))=\sum_{I\in \widetilde{\mathscr{D}}_k^1} \theta_I(t)=1,\quad \forall t\in \R.
\]

\medskip
{\flushleft \bf Step 3. Construct partition of unity on $\T^d$ and proof of Lemma~\ref{lemma:pou-scaling-Td}.}
\medskip

Now we can naturally define partition of unity on $\T^d$ by tensor product. Define 
\[
\varphi_k(\bz)=\varphi(z_1,z_2,\ldots,z_d)\coloneq \prod_{\ell=1}^d\rho_k(z_\ell),\quad \forall \bz=(z_1,z_2,\ldots,z_d)\in \T^d.
\]
We claim that $\varphi_k$ is the desired function in Lemma~\ref{lemma:pou-scaling-Td}. By definition, for any $\bt\in \R^d$, we have
\[
\sum_{\boldsymbol{I}\in \mathscr{D}_k^d} \varphi_{k}(\mathcal{E}(\bt-\boldsymbol{c}_{\boldsymbol{I}}))=\sum_{\boldsymbol{I}\in \mathscr{D}_k^d} \prod_{\ell=1}^d \rho_{I_\ell}(e(t_\ell-c_{I_\ell})),
\]
where $\boldsymbol{I}=I_1\times I_2\times\cdots \times I_d$ and $I_\ell\in \mathscr{D}_k^1$ for each $\ell=1,2,\ldots,d$. Then
\[
\sum_{\boldsymbol{I} \in \mathscr{D}_k^d} \prod_{\ell=1}^d \rho_{I_\ell}(e(t_\ell-c_{I_\ell}))=\prod_{\ell=1}^d \bigg[\sum_{I_\ell \in \mathscr{D}_k^1} \rho_{k}(e(t_\ell-c_\ell))\bigg]=1.
\]
This proves \eqref{eq:pou-Td-euclidean-coordinate} in Lemma~\ref{lemma:pou-scaling-Td}. Note that
\[
\supp \varphi_k=\prod_{\ell=1}^d \supp \rho_k\subset \prod_{\ell=1}^d e([-2^{-k},2^{-k}])=\mathcal{E}([-2^{-k},2^{-k}]^d),
\]
which is (2) in Lemma~\ref{lemma:pou-scaling-Td}.

Now we prove (3) in Lemma~\ref{lemma:pou-scaling-Td}. By definition, for any multi-index $\boldsymbol{\alpha}=(\alpha_1,\alpha_2,\ldots,\alpha_d)\in \N^d$ and $\bt=(t_1,t_2,\ldots,t_d)\in \R^d$, we have
\[
D^{\boldsymbol{\alpha}}\varphi_k(\mathcal{E}(\bt))=\prod_{\ell=1}^d \frac{\mathrm{d}^{\alpha_\ell}}{\mathrm{d}t_\ell^{\alpha_\ell}}\rho_k(e(t_\ell-c_{I_\ell})).
\]
But
\begin{align*}
\frac{\mathrm{d}^{\alpha_\ell}}{\mathrm{d}t_\ell^{\alpha_\ell}}\rho_k(e(t_\ell-c_{I_\ell}))&=\frac{\mathrm{d}^{\alpha_\ell}}{\mathrm{d}t_\ell^{\alpha_\ell}}\bigg[\sum_{h\in \Z} \theta(2^k(t_\ell-c_{I_\ell}-h))\bigg]
\\
&=\sum_{h\in \Z} \frac{\mathrm{d}^{\alpha_\ell}}{\mathrm{d}t_\ell^{\alpha_\ell}}\theta(2^k(t_\ell-c_{I_\ell}-h))\\
&=\sum_{h\in \Z} 2^{k\alpha_\ell}\bigg(\frac{\mathrm{d}^{\alpha_\ell}}{\mathrm{d}t^{\alpha_\ell}}\theta\bigg)(2^k(t_\ell-c_{I_\ell}-h)).
\end{align*}
Since summands have disjoint supports, we have
\[
\bigg\|D^{\boldsymbol{\alpha}}\varphi_k\bigg\|_{L^\infty(\T^d)}\le \prod_{\ell=1}^d 2^{k\alpha_\ell}\bigg\|\frac{\mathrm{d}^{\alpha_\ell}}{\mathrm{d}t^{\alpha_\ell}}\theta\bigg\|_{L^\infty(\R)}\le 2^{k|\boldsymbol{\alpha}|}\|\theta\|_{C^{|\boldsymbol{\alpha}|}(\R)}^d,
\]
which implies (3) in Lemma~\ref{lemma:pou-scaling-Td}. 

\newcommand{\etalchar}[1]{$^{#1}$}
	
\end{document}